\documentclass[11pt]{amsart}
\usepackage[utf8]{inputenc}
\usepackage{xcolor,amsmath,amssymb,amsfonts,hyperref,mathtools,tikz-cd,array,orcidlink}
\usepackage{amsthm}
\usepackage[section]{placeins}
\usepackage{fullpage}

\usepackage{multicol}
\usepackage[shortlabels]{enumitem}
\usepackage[misc]{ifsym}
\usepackage{algorithm}
\usepackage{algpseudocode}

\usepackage{tikz-cd}
\tikzcdset{scale cd/.style={every label/.append style={scale=#1},
    cells={nodes={scale=#1}}}}

\usepackage{comment}
\usepackage[capitalise]{cleveref}
\usepackage{csquotes}
\allowdisplaybreaks
\makeatletter
\algnewcommand{\linecomment}[1]{\Statex \hskip\ALG@thistlm \textcolor{blue}{/*#1*/}}
\makeatother
\makeatletter
\def\verbatim@nolig@list{\do\`\do\'}% no comma
\makeatother

\newtheorem{theorem}{Theorem}[section]
\newtheorem*{theorem*}{Main Theorem}
\newtheorem{proposition}[theorem]{Proposition}
\newtheorem{corollary}[theorem]{Corollary}
\newtheorem{lemma}[theorem]{Lemma}

\theoremstyle{definition}
\newtheorem{definition}[theorem]{Definition}

\newtheorem{notation}[theorem]{Notation}

\newtheorem{remark}[theorem]{Remark}

\newcommand{\orcid}[1]{\href{https://orcid.org/#1}{\textcolor[HTML]{A6CE39}{\aiOrcid}}}

\newcommand{\B}{\mathcal{B}}
\newcommand{\Bs}{\mathcal{B}^{\text{sing}}}
\newcommand{\C}{\mathcal{C}}

\newcommand{\A}{\mathcal{A}}

\renewcommand{\O}{\mathcal{O}}
\renewcommand{\H}{\mathcal{H}}

\newcommand{\Aut}{\text{Aut}}
\newcommand{\PSp}{\text{PSp}}

\newcommand{\FF}{\mathbb{F}}

\newcommand{\ZZ}{\mathbb{Z}}

\let \H\HH

\newcommand{\CC}{\mathbb{C}}
\newcommand{\PP}{\mathbb{P}}

\newcommand{\cube}{\odot^3}

\newcommand{\diag}{\mathop{diag}}

\newcommand{\DFT}{M_{3,3}}

\definecolor{lightgreen}{rgb}{0.86, 0.93, 0.78}
\definecolor{bordergreen}{rgb}{0.55, 0.76, 0.74}
\definecolor{lightblue}{rgb}{0.70, 0.90, 0.99}
\definecolor{borderblue}{rgb}{0.01, 0.66, 0.96}
\definecolor{lightamber}{rgb}{1, 0.93, 0.70}
\definecolor{borderamber}{rgb}{1, 0.76, 0.03}
\definecolor{lightcolor4}{rgb}{ 0.93, 0.70, 1}
\definecolor{bordercolor4}{rgb}{0.76, 0.03, 1}
\definecolor{lightcolor5}{rgb}{0.78,0.86,0.93}
\definecolor{bordercolor5}{rgb}{0.74,0.55,0.76}
\RequirePackage{totcount}
\RequirePackage{color}
\RequirePackage[colorinlistoftodos]{todonotes}
\usepackage{orcidlink}

\newtotcounter{notecount}
\title{Abelian surfaces in Hesse form\\ and explicit isogeny formulas}

\author{Thomas Decru${}^{\orcidlink{0000-0003-0253-4180}}$, Sabrina Kunzweiler${}^{\orcidlink{0000-0002-6179-2094}}$}

\begin{document}
	\begin{abstract}
	We develop a new method for the computation of $(3,3)$-isogenies between principally polarized abelian surfaces. The idea is to work with models in $\PP^8$ induced by a symmetric level-$3$ theta structure. In this setting, the action of three-torsion points is linear, and the isogeny formulas can be described in a simple way as the composition of easy-to-evaluate maps. In the description of these formulas, the relation with the Burkhardt quartic threefold plays an important role. Furthermore, we discuss generalizations of the idea to higher dimensions as well as different isogeny degrees.
	\end{abstract}
	\maketitle

	\section{Introduction}

For explicit computations with elliptic curves, one often works with a plane projective model. Concretely, given an elliptic curve $E$ with distinguished point $\O$, an embedding into the projective plane is defined by the choice of a basis for $L(3\O)$. A typical choice is $\{1,x,y\}$ where $x$ has a pole of order two and $y$ a pole of order three at $\O$. In this setting, the image of $E$ under the embedding
\begin{align*}
E \hookrightarrow \PP^2, \quad P &\mapsto (x(P):y(P):1), ~~ P \neq \O,\\
\O & \mapsto (0:1:0),
\end{align*}
is defined by a Weierstra\ss\ equation.
In the following, we are interested in a different choice of coordinates for $L(3\O)$. 

From now on, assume that we are working over a field with characteristic $p \neq 3$.
We know that the $3$-torsion points of $E$ act by linear transformation on the coordinates of a point. Denoting $E[3] = \langle P, Q\rangle$, we choose a basis so that the corresponding embedding $\iota: E \to \PP^2$ satisfies
\[
\iota(R) + \iota(P) = \begin{pmatrix} 0 & 1 & 0\\ 0&0&1\\ 1&0&0 \end{pmatrix} \cdot \iota(R),\quad 
\iota(R) + \iota(Q) = \begin{pmatrix} 1 & 0 & 0\\ 0&\omega&0\\ 0&0&\omega^2 \end{pmatrix} \cdot \iota(R),
\]
for any point $R \in E$. Here $\omega = e_3(Q,P)$ denotes the Weil pairing of $P$ and $Q$; in particular $\omega$ is a primitive third root of unity. The image of $E$ in $\PP^2$ is in {\em Hesse form}. That is there exists a projective parameter $d = (d_0:d_1)$ with $d_0(d_0^3-d_1^3) \neq 0$ so that $\iota(E) = \H_d$ with
\begin{equation} \label{eq:ell-hesse}
    \H_d: d_1(X^3 + Y^3 + Z^3) = 3d_0 XYZ.
\end{equation}
More formally, this embedding is induced by a symmetric theta structure of level $3$, and in particular, the coordinate functions $X,Y,Z$ correspond to level-$3$ theta functions, \cite[Section 5, Case (b)]{mumford1966equations}. Note that this model is defined over the base field if $E$ has full rational $3$-torsion. In the literature, the neutral element is usually chosen to be $0_{\H_d} = (0:-1:1)$. In this case, the $3$-torsion basis generators from above are of the form
\begin{equation} \label{eq:ell-hesse-torsion}
    P = (-1: 1 :0) ,\quad  Q = (0: -\omega^2 : 1) \quad \in \quad \H_{d}.
\end{equation}

In this work, our focus is on principally polarized (p.p.) abelian surfaces. Throughout the paper, we work over a field $k$ with characteristic $p\neq 2,3$ which contains a primitive third root of unity $\omega$.  The construction above generalizes to higher dimension.
In particular, for a p.p. abelian surface with full rational $3$-torsion, we  can consider the embedding  
\[
\iota: A \hookrightarrow \PP^8
\]
induced by a symmetric theta structure of level $3$. As in the elliptic curve case, the action of $3$-torsion points in this model is normalized. Explicit equations for the image $\A_{d,h} := \iota(A)$ in $\PP^8$ are known. More precisely, $\A_{d,h}$ is described as the zero locus of $8$ quadrics and $4$ cubic equations, \cite{gunji2006defining}.  In our notation $d =(d_0: \dots:d_4) \in \PP^4$ is a coefficient vector defining the cubic equations, and $h = (h_0:\dots: h_4) \in \PP^4$ defines the quadrics (see Theorem \ref{thm:defining-equations-dim2} for a precise description). In analogy with the elliptic curve setting, we say that $\A_{d,h}$ is a principally polarized abelian surface in {\em Hesse form}.

\subsection{Contributions}
The main result of our paper is a new method for the computation of $(3,3)$-isogenies between p.p. abelian surfaces in Hesse form. Here, the prefix $(3,3)$ indicates that the kernel of the isogeny is a maximal $3$-isotropic group. The isogeny evaluation will be decomposed into four simple operations (see Notation \ref{not:basic-operations} for more formal definitions):
\begin{itemize}
    \item $S_K$: a symplectic transformation depending on the input kernel $K$,
    \item $\cube$: coordinate-wise cubing,
    \item $M_{3,3}$: a discrete Fourier transform (which is a symplectic transformation as well),
    \item $C_\lambda$: coordinate-wise (symmetric) scaling by the vector $\lambda \in \PP^4$.
\end{itemize}

Using this notation, our results can be summarized as follows.
\begin{theorem*}[informal]
    Let $\A_{d,h}$ be an abelian surface in Hesse form. On input a maximal isotropic group $K = \langle Q_1,Q_2\rangle \subset \A_{d,h}[3]$ and $9$-torsion points $R_1,R_2$ satisfying $3\cdot R_1 = Q_1$, $3\cdot R_2 = Q_2$, we can determine a symplectic transformation $S_K$, and a vector $\lambda \in \PP^4$ 
    so that 
    \[
    \phi = (C_\lambda \circ M_{3,3} \circ \cube \circ S_K): \A_{d,h} \to \A_{d',h'},
    \]
    describes a $(3,3)$-isogeny with kernel $K$, between p.p. abelian surfaces in Hesse form.
\end{theorem*}

Ignoring the cost of additions and scalar multiplications by a third root of unity $\omega$, the maps $S_K$ and $M_{3,3}$ are essentially free. The cost of an isogeny evaluation is then given by $9$ cubings (application of  $\cube$), and  $9$ multiplications (application of $C_\lambda$) over the field of definition. 

We remark that this method allows us to evaluate multiplication by $3$ on $\A_{d,h}$ as the composition of a $(3,3)$-isogeny with its dual (provided that we are given two auxiliary $9$-torsion points as in the theorem).
Furthermore, we highlight that a priori the method  works for general p.p. abelian surfaces in Hesse form. In particular, so-called gluing, splitting and diagonal isogenies are included. However, the concrete computation of the coefficient vector $\lambda$ has to be adapted when working with certain edge cases.

On a technical side, our method is based on classical results about theta structures and the moduli space of abelian surfaces of level $3$. An important ingredient is the explicit description of $\A_{d,h}$ by Gunji \cite{gunji2006defining} which is based on the description of quadratic relations of level-$3$ theta functions by Coble \cite{coble1917point}, and cubic relations by Birkenhake and Lange \cite{birkenhake1990cubic}. Furthermore, the explicit relation of the model $\A_{d,h}$ with the Burkhardt quartic and its dual play an important role. For this viewpoint, we rely on an analysis of the moduli space by Freitag and Manni \cite{freitag2004burkhardt1}, \cite{freitag2004burkhardt2},  Hunt's exposition on the geometry of the Burkhardt quartic \cite{hunt2006burkhardt}, and further results by Bruin and Nasserden on the arithmetic of the Burkhardt quartic  \cite{bruin2018arithmetic}.

Our article is accompanied by a SageMath \cite{sagemath} package implementing the described methods and different functionality concerning the arithmetic of abelian surfaces in Hesse form:
\begin{quote}
    \url{https://github.com/sabrinakunzweiler/level-3-arithmetic}
\end{quote}

\subsection{Comparison with isogeny formulas in the literature}

Our method is inspired by $(2,2)$-isogeny formulas on Kummer surfaces using the level-$2$ theta structure, \cite{gaudry2007fast,damien-notes,level-2-theta-dimension-2}. In that setting, addition by $2$-torsion points acts by linear transformation on the coordinates of a point. This has led to remarkably simple isogeny formulas. Similar to our results, a $(2,2)$-isogeny is decomposed into four simple operations: a symplectic transformation, coordinate-wise squaring, a Hadamard transform, and coordinate-wise scaling. Note that the Hadamard transform is a discrete Fourier transform (with respect to a primitive square root of unity); and clearly cubing is the natural analogue of squaring. In that sense our method to compute $(3,3)$-isogenies can be seen as the natural analogue.

We note that there exist general purpose algorithms to compute $(\ell,\ell)$-isogenies in the level-$2$ (or level-$4$) theta model developed in a series of work by Cosset, Lubicz and Robert \cite{lubicz2012computing,cosset2015computing,lubicz2023fast}, see also \cite{yoshizumi2025efficient} for concrete optimizations for different values of $\ell$. While these algorithms are also based on the theory of theta functions, they are different in nature. In particular, working with even level, the translation by $3$-torsion points is not given by linear transformations, hence this symmetry cannot be used to derive isogeny formulas.

Previous specialized algorithms to compute $(3,3)$-isogenies are based on different ideas. Explicit formulas to compute $(3,3)$-isogenies between Jacobians of genus-$2$ curves were first given by Bruin, Flynn and Testa \cite{3descent} in the context of $(3,3)$-descent. Some algebraic simplifications of the formulas were later provided in \cite{decru2023efficient} which made them suitable in a cryptanalytic context. An advantage of the approach by Bruin, Flynn and Testa is that they only require the kernel of the $(3,3)$-isogeny to be rational, while  our approach requires full rational $3$-torsion. However, the evaluation of the formulas is considerably more expensive than in our new framework.
The most efficient known method to compute $(3,3)$-isogenies is an algorithm by Corte-Real Santos, Costello and Smith \cite{corte2025efficient} working on so-called {\em fast} Kummer surfaces, and requiring full rational $2$-torsion. In their optimized formulas, an isogeny evaluation costs $4$ squarings and $26$ multiplications while in our approach, a $(3,3)$-isogeny is evaluated with $9$ squarings and $18$ multiplications. This makes our new method (slighly) faster. But it should be mentioned that in turn, the tripling formula on the Kummer surface is more efficient.

The mentioned methods for the computation of $(3,3)$-isogenies all work on (different models of) the Kummer surface. In contrast to this, in our approach, we directly work on the abelian surface itself. Moreover, our method is not restricted to working with irreducible p.p. abelian surfaces. But the description of p.p. abelian surfaces in Hesse form, $\A_{d,h}$, naturally includes products of elliptic curves.

\subsection{Applications and perspectives}

Explicit isogeny computations are important for algorithmic number theory. For instance, $(3,3)$-isogenies have been used in descent \cite{3descent} or in the construction of Jacobians of genus-$2$ curves with points of large order \cite{howe-point-counting}.  Moreover, $(3,3)$-isogenies have recently found applications in post-quantum cryptography \cite{SIDH-attack}. Our method is particularly well-suited for the computation of $(3^n,3^n)$-isogenies for some large integer $n$. Decomposing the computation as a chain of $(3,3)$-isogenies, the auxiliary $9$-torsion points are naturally available at each step. For instance, our method could be applied in isogeny-based cryptography, more precisely in the two-dimensional variants of the signature scheme SQIsign \cite{SQIPrime,SQISign2D-West,SQISign2D-East}, one of which is in the running to become standardized by NIST \cite{sqisign}. One of the bottleneck computations in this scheme is the evaluation of a $(2^n,2^n)$-isogeny which could (after some parameter changes) be replaced by a $(3^n,3^n)$-isogeny. % In our implementation, we apply our new formulas to compute such a $(3^n,3^n)$-isogeny between a product of elliptic curves.

Apart from these concrete applications, our results more generally show that the $\PP^8$-model of a p.p. abelian surface is well-suited for explicit computations. Despite the similarities with $\PP^2$-models for elliptic curves, we are not aware of many computational applications  of the former in the literature.

An interesting question is whether our method generalizes to arbitrary dimension. We briefly elaborate on this question in the last section of the paper, but leave a concrete analysis of specific dimensions as future work.\footnote{It is known that the analogous  $(2,2)$-isogeny formulas hold in arbitrary dimension \cite{damien-notes}. And in practice, they have been implemented for different purposes in dimensions $1,2,3$ and $4$ \cite{sarkis2024computing,level-2-theta-dimension-2,kunzweiler2025radical,level-2-theta-dimension-4}.} In the preparation of this article, we verified that the computation of $3$-isogenies between elliptic curves in Hesse form is completely analogous. More precisely, a $3$-isogeny $\phi: \H_d \to \H_{d'}$ with kernel $K$ can be evaluated as 
\[
\phi = (C_\lambda  \circ M \circ \cube \circ S_K)(P),
\]
where similar as before $C_\lambda: \PP^3 \to \PP^3$ denotes coordinate-wise (symmetric) scaling, $M$ is discrete Fourier transform, $\cube$ coordinate-wise cubing, and $S_K$  a symplectic transformation.  A short article presenting these formulas, restricted to elliptic curves, is currently in submission for the proceedings volume of a cryptographic conference, where we presented this phenomenon \cite{hessiantripling}.  Given that $3$-isogeny formulas between elliptic curves in Hesse form are well known, the proofs in this case are elementary and independent of the theory developed in this manuscript. However, the formulas in the elliptic curve case immediately follow from our main result, applied to a product of elliptic curves in Hesse form. 

Another future direction of research is the translation of our formulas to the Kummer surface of $\A_{d,h}$. On the one hand, under the projection $\A_{d,h} \to \A_{d,h}/\langle \pm 1\rangle$ some of the symmetry concerning the action of $3$-torsion points gets lost. On the other hand there are fewer coordinates to work with. It would be interesting to see if the resulting formulas for the computation of $(3,3)$-isogenies, and in particular tripling, become more efficient.

\subsection{Outline}
In Section \ref{sec:burkhardt}, we recall properties of the Burkhardt quartic, its dual and its Hessian. Section \ref{sec:hesse-form} is dedicated to the description of p.p. abelian surfaces in Hesse form. In particular, we analyse the relation with the Burkhardt quartic, symplectic transformations, and the special case of product surfaces.  The $3$-torsion points and their action on $\A_{d,h}$ are described in Section \ref{sec:three-torsion}, and Section \ref{sec:twisted-surface} includes a short description of twists of $\A_{d,h}$. Section \ref{sec:isogeny} is dedicated to $(3,3)$-isogenies.  We prove our main result (Theorem \ref{thm:3-isogeny-dim2}), and explain adaptations to obtain a completely general method. This is further illustrated by a detailed example in Section \ref{sec:example}. Finally, in Section \ref{sec:generalizations}, we discuss generalizations of our method to higher dimensions and different isogeny degrees.

\subsection{Acknowledgements}

We would like to thank Damien Robert for his very helpful and insightful feedback throughout this project. Thomas Decru is supported by Fonds voor Wetenschappelijk Onderzoek (FWO) with reference number 1245025N, by the European Research Council (ERC) under the European Union’s Horizon 2020 research and innovation programme (grant agreement ISOCRYPT – No. 101020788), by the Research Council KU Leuven grant C14/24/099 and by CyberSecurity Research Flanders with reference number VOEWICS02. Sabrina Kunzweiler received funding from the French National Research Agency (ANR) under the ANR CIAO with reference ANR-19-CE48-0008, and the France 2030 program under grant ANR-22-PETQ-0008 (PQ-
TLS).
    \section{The Burkhardt quartic} \label{sec:burkhardt}

The {\em Burkhardt quartic} is a hypersurface in $\PP^4$ defined by the quartic equation
\begin{equation}
	\B: F(x_0,x_1,x_2,x_3,x_4) = x_0(x_0^3 + x_1^3 + x_2^3 + x_3^3 + x_4^3) + 3 x_1x_2x_3x_4 = 0 \subset \PP^4. \label{eq:burkhardt-quartic}
\end{equation}

The Burkhardt quartic is birationally equivalent to the moduli space of principally polarized abelian surfaces with a full level-$3$ structure  over $\ZZ[1/3, \omega]$, \cite{van1987note}. In particular, to any nonsingular point $h=(h_0: \dots: h_4) \in \B$, one can associate an irreducible p.p. abelian surface with marked level-$3$ structure. An explicit construction is provided in \cite{bruin2018arithmetic}. 

\subsection{Automorphism group} The automorphism group of the Burkhardt quartic is isomorphic to the projective symplectic group $\PSp_4(\FF_3)$. We use the representation $\Aut(\B)  = \langle M, S_0,S_1,S_2\rangle$ with
\begin{equation}
	\label{eq:automorphisms}
	\begin{aligned}
	M =&~ \begin{pmatrix}
		1 & 2 & 2 & 2 & 2 \\
		1 & -1 & 2 & -1 & -1 \\
		1 & 2 & -1 & -1 & -1 \\
		1 & -1 & -1 & -1 & 2 \\
		1 & -1 & -1 & 2 & -1
	\end{pmatrix},
	\end{aligned}
	\quad 
	\begin{aligned}
	S_0 =&~ \diag(1,\omega,1,\omega, \omega)\\
	S_1 =&~ \diag(1,1,1,\omega^2, \omega)\\
	S_2 =&~ \diag(1,1,\omega,\omega, \omega).
	\end{aligned}
\end{equation}
from \cite[Proposition 1]{freitag2004burkhardt1}. Here, $\diag(\alpha_0, \dots, \alpha_4)$ denotes the diagonal matrix with entries $\alpha_0, \dots, \alpha_4$ on the diagonal. The action on points of the quartic is given by the right action on the coordinate vector; we simply denote $h \cdot g$ for $h \in \B$ and $g \in \langle M, S_0,S_1,S_2\rangle$.

\begin{remark}
    The automorphisms defined by $S_0,S_1,S_2$ act by scaling the coordinates of a point on the Burkhardt quartic by third roots of unity in a compatible way. We note that
    \[
    \langle S_0,S_1,S_2 \rangle = \{\diag(1,\omega^a, \omega^b, \omega^c, \omega^{2(a + b + c)}) \mid a,b,c \in \{0,1,2\} \}.
    \]
    Furthermore, we will see later that the automorphism defined by $M$ is related to a discrete Fourier transform (Definition \ref{def:DFT}).
\end{remark}

In our applications, we will be interested in reducible abelian surfaces, as well. For this purpose, we first recall a result about the singular points of $\B$. 

\begin{lemma} \label{lem:singularities-burkhardt}
	The singular locus of the Burkhardt quartic, $\Bs$, consists of  $45$ singular points. Explicitly, the singularities $h = (h_0:h_1:h_2:h_3:h_4) \in \B$ are given by
	\[
	(a) \quad h = \sigma (0: 0: 0: -\omega^i:1), \quad \text{with } i \in \{0,1,2\}
	\]
	and $\sigma \in \mathcal S_4 \subset \mathcal S_5$ a permutation fixing $h_0= 0$,
	\[
	(b) \quad h = (1 : -\omega^i : - \omega^j : -\omega^k : -\omega^{-i-j-k}) \quad \text{with } i,j,k \in \{0,1,2\}. 
	\]
	Moreover, $\Aut(\B)$ acts transitively on $\Bs$. 
\end{lemma}

\begin{proof}
	A case distinction between $h_0 = 0$ and $h_0 \neq 0$ yields the explicit description of the singularities of $\B$. There are $18$ singularities  of type $(a)$, and $27$ of type $(b)$. 
	
	The transitivity of the action can be verified explicitly. Concretely, for any singular point $h \in \B$, we computed an element $g \in \Aut(\B)$ with $h\cdot g = (0:0:0:-1:1)$; see \cite{source_code}.
\end{proof}

\subsection{Burkhardt dual}\label{subsec:burkhardt-dual} The  Burkhardt quartic is only birationally equivalent to the moduli space of principally polarized abelian surfaces with level-$3$ structure. The Satake compactification of the moduli space is given by the normalization of the {\em dual of the Burkhardt quartic}, \cite{freitag2004burkhardt1}. The latter will allow us to describe reducible abelian surfaces.

\begin{definition}
	The {\em Burkhardt dual} $\tilde{\B}$ is the dual of the Burkhardt quartic $\B$, i.e, it is a hypersurface in $\PP^4$ defined by $\tilde{F}$, where $\tilde{F}$ is the dual of the polynomial $F$ defining ${\B}$.
\end{definition}

 The polynomial $\tilde{F}$ is homogeneous of degree $18$. An explicit description is provided in \cite[Proposition 3]{freitag2004burkhardt1}.  We further recall that by the definition of the dual, $\tilde{F}$ is the minimal polynomial satisfying
\[
\tilde{F}(\nabla F) \cong 0 \pmod{F},
\]
where $\nabla$ denotes the gradient operator. In particular, $\nabla$ defines an embedding $\nabla: \B \setminus \Bs \to \tilde{\B}$. Explicitly, 
\begin{equation}
	\label{eq:gradient-embedding}
	\nabla : (x_0: \dots: x_4) \mapsto \begin{pmatrix}
		4 x_0^3 + x_1^3 + x_2^3 + x_3^3 + x_4^3\\
		3(x_0x_1^2 + x_2x_3x_4)\\
		3(x_0x_2^2 + x_1x_3x_4)\\
		3(x_0x_3^2 + x_1x_2x_4)\\
		3(x_0x_4^2 + x_1x_2x_3)
	\end{pmatrix} .
\end{equation}
The action of $\Aut(\B)$ extends to an action on the dual, where the elements $g \in \Aut(\B)$ act by multiplication on the left with $g^{-1}$; we denote $g^{-1} \cdot d$ for $d \in \tilde{\B}$. Note that, up to multiplication by a scalar, $M \equiv M^{-1}$ for the matrix $M$ defined above.

\subsection{Hessian of the Burkhardt quartic} Another variety that appears when studying the moduli space of abelian surfaces with level-$3$ structure is the Hessian of the Burkhardt quartic. 

\begin{definition}\label{def:hessian}
    The {\em Hessian of the Burkhardt quartic} $H(\B)$ is the projective hypersurface in $\PP^4$ defined by 
    \[
    H(\B): H(F) := \frac{1}{486} \det\left( \frac{\partial F}{\partial x_i \partial x_j}\right) = 0,
    \]
    where $F$ is the polynomial defining $\B$. 
\end{definition}

Note that the scaling ensures that $H(F)$ is integral and has content $1$ (see for example \cite[Section 3]{bruin2018arithmetic}).  One can check that $H(F)$ is homogenous of degree $10$. 

    \section{Principally polarized abelian surfaces in Hesse form} \label{sec:hesse-form}

Let $A$ be a principally polarized abelian surface. Consider a line bundle $L$ inducing the principal polarization. Then $L^{\otimes 3}$ is very ample, and defines an embedding $\iota:A \hookrightarrow\PP^8$. For any $3$-torsion point $T \in A[3]$, the translation map $P \mapsto P + T$ acts by linear transformation on the coordinates of $\iota(A)$. 
Here, we will consider a specific embedding $\Theta:A \to \PP^8$ induced by a symmetric theta structure of level $3$. In this setting, the action by $3$-torsion points is {\em normalized}.  

We first recall an explicit description of the model $\Theta(A) \subset \PP^8$. Then explain the relation of this model with points on the Burkhardt quartic and its dual in Subsection \ref{subsec:relations-burkhardt}, discuss symplectic transformations in Subsection \ref{subsec:symplectic-trafos}, and analyze the reducible setting in more detail in Subsection \ref{subsec:product}. The action of three-torsion points is discussed in Section \ref{sec:three-torsion}.

\begin{theorem}[Theorem 8.3 in \cite{gunji2006defining}] \label{thm:defining-equations-dim2}
    Let $A$ be a principally polarized abelian surface with a level-$3$ structure. Then there exist coefficient vectors $d=(d_0: \dots: d_4), h = (h_0:\dots:h_4) \in \PP^4$, and an embedding $\iota: A \to \PP^8$ with
    \[
    \iota(A) = V(F_1,F_2,F_3,F_4, G_0, \dots, G_8),
    \] where
    the $F_i$ are the cubic polynomials
    \begin{eqnarray*}
    F_1 &=& d_1\sum_{j=0}^8 X_j^3 - 3d_0 (X_0X_1X_2 + X_3X_4 X_5 + X_6X_7X_8),\\
    F_2 &=& d_2\sum_{j=0}^8 X_j^3 - 3d_0 (X_0X_3X_6 + X_1X_4 X_7 + X_2X_5X_8),\\
    F_3 &=& d_3\sum_{j=0}^8 X_j^3 - 3d_0 (X_0X_4X_8 + X_1X_5X_6 + X_2X_3X_7),\\
    F_4 &=& d_4\sum_{j=0}^8 X_j^3 - 3d_0 (X_0X_5X_7 + X_1X_3X_8 + X_2X_4X_6),
    \end{eqnarray*}
    and the $G_i$ are the quadratic polynomials 
    %obtained by applying the nine permutations of $\langle \sigma_1, \sigma_2 \rangle$
    %to the polynomial
    %\[
    %Q = h_0 X_0^2 + h_1 X_1 X_2 + h_2 X_3 X_6 + h_3 X_4 X_8 + h_4 X_5 X_7,
    %\]
    %where
    %\[
    %\sigma_1 = (0\,1\,2)(3\,4\,5)(6\,7\,8),\quad \sigma_2 = (0\,3\,6)(1\,4\,7)(2\,5\,8).
    %\]
    %More precisely, we let $G_i = Q(\sigma_i^ (X_0, \dots, X_8))$.  
    \begin{eqnarray*}
        G_0  &=& h_0 x_0^2 + h_1 x_1 x_2 + h_2 x_3 x_6 + h_4 x_5 x_7 + h_3 x_4 x_8,\\
        G_1 &=& h_0x_1^2 + h_1 x_0 x_2 + h_2 x_4 x_7 + h_3 x_5 x_6 + h_4 x_3 x_8,\\
        G_2 &=& h_0 x_2^2 + h_1x_0 x_1 + h_2 x_5 x_8 + h_3 x_3 x_7 + h_4 x_4 x_6,\\
        G_3 &=& h_0x_3^2 + h_1 x_4 x_5 + h_2 x_0 x_6 + h_3 x_2 x_7 + h_4 x_1 x_8,\\
        G_4 &=& h_0x_4^2 + h_1 x_3 x_5 + h_2 x_1 x_7 + h_3 x_0 x_8 + h_4 x_2 x_6,\\
        G_5 &=& h_0 x_5^2 + h_1x_3 x_4 + h_2 x_2 x_8 + h_3 x_1 x_6 + h_4 x_0 x_7,\\
        G_6 &=& h_0 x_6^2 + h_1 x_7 x_8 + h_2x_0 x_3 + h_3 x_1 x_5 + h_4 x_2 x_4,\\
        G_7 &=& h_0 x_7^2 + h_1 x_6 x_8 + h_2 x_1 x_4 + h_3x_2 x_3 + h_4 x_0 x_5,\\
        G_8 &=& h_0 x_8^2 + h_1 x_6 x_7 + h_2 x_2 x_5 + h_3 x_0 x_4 + h_4x_1 x_3.
    \end{eqnarray*}
\end{theorem}

We note that the description in \cite[Theorem 8.3]{gunji2006defining} does not include the polynomial $F_4$, since it is contained in the ideal generated by the remaining polynomials. However, in our setting, it will be more natural to include this polynomial in the description. Furthermore, we point out that the cited theorem is more precise in the sense that it provides an explicit description of the coefficients. However, for our purposes, the above statement is enough.

\begin{remark}
    In \cite{gunji2006defining}, the author uses the classic analytic theory of theta functions over $\CC$. Using Mumford's algebraic theory of theta functions, these results can be translated to arbitrary fields with characteristic $p>3$. The restriction on the characteristic comes from the fact that the coefficients of the defining equations are described by theta functions of level $18$, hence the equations are well-defined over $\ZZ[1/{18}]$.

    On the other hand, the coordinate functions $x_0, \dots , x_8$ are only of level $3$. It would be interesting to see if the results could be translated to $p=2$, as well.
    We refer to \cite[Section 2.3.6]{robert2021theory} for more details on lifting arguments concerning level-$n$ theta structures.
\end{remark}

\begin{definition} \label{def:hessian-form-av}
    We say that a principally polarized abelian surface $A$ is in {\em Hesse form} if
    \(
    A = V(F_1,F_2,F_3,F_4,G_0,\dots, G_8)
    \) with $F_i, G_i$ as in Theorem  \ref{thm:defining-equations-dim2}. In this case, we denote $A = \A_{d,h}$ where $d = (d_0:d_1:d_2:d_3:d_4)$ and $h = (h_0: h_1:h_2:h_3:h_4)$ with the notation from Theorem \ref{thm:defining-equations-dim2}.
\end{definition}

We use the terminology {\em Hesse form} as this form can be seen as an analogue to the Hesse form of elliptic curves. 

\begin{remark}
 If $\A_{d,h}$ is irreducible, then the variety is already defined by the quadrics, i.e. $V(F_1,F_2,F_3,F_4,G_0,\dots, G_8) = V(G_0,\dots, G_8)$. This was already claimed by Coble \cite{coble1917point} and proven by Barth \cite{barth1995quadratic}. The cubic relations are necessary to include the degenerate cases. The explicit description by Gunji \cite{gunji2006defining} that we use here, builds on work by Birkenhake and Lange on cubic theta relations \cite{birkenhake1990cubic}.
\end{remark}

\subsection{The $2$-torsion} \label{subsec:two-torsion}

Multiplication by $[-1]$ acts as a linear transformation on the coordinates of a point in $\A_{d,h}$. More precisely, it acts by the permutation $(0)(1\,2)(3\,6)(4\,8)(3\,7)$ on the coordinates of a point, i.e.
\[
    \begin{tikzcd} [row sep = 15pt] 
    ~~~P = (x_0:x_1:x_2:x_3:x_4:x_5:x_6:x_7:x_8) \dar[maps to]\\
    - P = (x_0:x_2:x_1:x_6:x_8:x_7:x_3:x_5:x_4).
    \end{tikzcd}
\]

\begin{remark}
    We recall that the coordinates $x_0, \dots, x_8$ correspond to level-$3$ theta functions. Using a ternary representation of the indices, i.e. 
    \[
        x_0 = x_{00},\; x_1 = x_{01} , \;\dots 
        \;, \; x_8 = x_{22},
    \]
    multiplication by $[-1]$ is given by $(x_i)_{i\in (\ZZ/3\ZZ)^2} \mapsto (x_{-i})_{i\in (\ZZ/3\ZZ)^2}$, as usual. We note that this notation is closer to the notation used in \cite{gunji2006defining}, in particular see \cite[Section 3]{gunji2006defining} for the relation with theta functions.
\end{remark}

The $2$-torsion points on $\A_{d,h}$ are the points that are fixed under multiplication by $[-1]$. These points come in two flavors: $\A_{d,h}[2]= U_1 \cup U_2$ with
\begin{eqnarray*}
	U_1 & = & \A_{d,h} \cap \{x_1 = x_2, x_3 = x_6, x_4 = x_8, x_5 = x_7\}, \\
	U_2 & = & \A_{d,h} \cap \{x_0 = 0, x_1 = -x_2, x_3 = -x_6, x_4 = -x_8, x_5 = -x_7\}.
\end{eqnarray*}
The elements in $U_1$ correspond to the eigenvalue $+1$ (of the multiplication by $[-1]$ map) and are called {\em even $2$-torsion points}; the elements in $U_2$ correspond to the eigenvalue $-1$ and are called {\em odd $2$-torsion points}. In total, there are $10$ even and $6$ odd $2$-torsion points \cite[Section 5.3.1]{hunt2006burkhardt}.

The neutral element can be either even or odd depending on whether the underlying theta structure is even or odd. Throughout this work, we assume that the underlying theta structure is even, hence the neutral element belongs to the set $U_1$. We denote
\begin{equation} \label{eq:neutral-element}
	0_{\A_{d,h}}  = (t_0 : t_1 : t_1 : t_2 : t_3: t_4 : t_2 : t_4 : t_3),
\end{equation}
for some $t_0, \dots, t_4$. 

This choice is natural for our applications since we typically construct $\A_{d,h}$ which is isogenous to a product of elliptic curves $E_1 \times E_2$. The latter comes equipped with the product theta structure - which is even if both $E_1$ and $E_2$ are equipped with the same type of theta structure (i.e. both even, or both odd).\footnote{In the literature, we only encountered examples with an odd $2$-torsion point as neutral element, see for example Theorem 3.14(b) in \cite{gruson2015alternating}. We thank Damien Robert for shedding light on this matter.}

Finally, we note that the coordinates of the neutral element are related to the geometry of the Burkhardt quartic. It holds that
\begin{equation} \label{eq:hessian-relation}
	(t_0 : 2t_1 : 2t_2 : 2t_3 : 2t_4) \in H(\B),
\end{equation}
where $H(\B)$ is the Hessian of the Burkhardt quartic (Definition \ref{def:hessian}). Indeed, this holds for all even $2$-torsion points. This relation can be obtained by studying the Kummer surface $\A_{d,h}/\langle \pm 1\rangle \subset \PP^4$ with coordinates $(x_0:x_1+x_2:x_3+x_6:x_4+x_8:x_5+x_7)$, and observing that the ten even $2$-torsion points correspond to the nodes of this Kummer surface,  cf. \cite[Section 5.3]{hunt2006burkhardt}. 

\subsection{Relation with the Burkhardt quartic} \label{subsec:relations-burkhardt}

We have already seen that the coordinates of the neutral element of $\A_{d,h}$ are related to a point on the Hessian of the Burkhardt quartic $\B$. Furthermore, it is well known that the coefficient vector $h$ of the quadrics defining $\A_{d,h}$ lie on the Burkhardt quartic itself:
\begin{equation} \label{eq:h-on-B}
h = (h_0: \dots :h_4) \in \B,
\end{equation}
see for example \cite[Section 5.3.1]{hunt2006burkhardt}.  Moreover, there is a relation with the coefficient vector $d$:
\begin{equation*}
h_0 d_0 + h_1 d_1 + h_2 d_2 + h_3 d_3 + h_4 d_4 = 0,
\end{equation*}
\cite[Equation 8.3]{gunji2006defining} which implies that $d$ lies on the dual of the Burkhardt quartic:
\begin{equation} \label{eq:d-on-Bd}
d = (d_0:\dots:d_4) \in \tilde{\B}.
\end{equation}
This is natural in the light of the description of the moduli space in \cite{freitag2004burkhardt1}, see also Subsection \ref{subsec:burkhardt-dual}. 

In the following, we provide explicit formulas for computing the coefficient vectors $d$ and $h$, given the neutral element $0_{\A_{d,h}}$ of an abelian surface in Hesse form $\A_{d,h}$. These formulas define maps between the Hessian of the Burkhardt quartic, its dual, and the Burkhardt quartic itself. We summarize the relations in Figure \ref{fig:relations-coefficients}. 

Our proofs of these relations are explicit, and only require knowledge of the defining equations. However, we note that they are induced by cubic relations of theta functions described in \cite{birkenhake1990cubic} which had been used to find the explicit equations for $\A_{d,h}$ in the first place \cite{gunji2006defining}. 

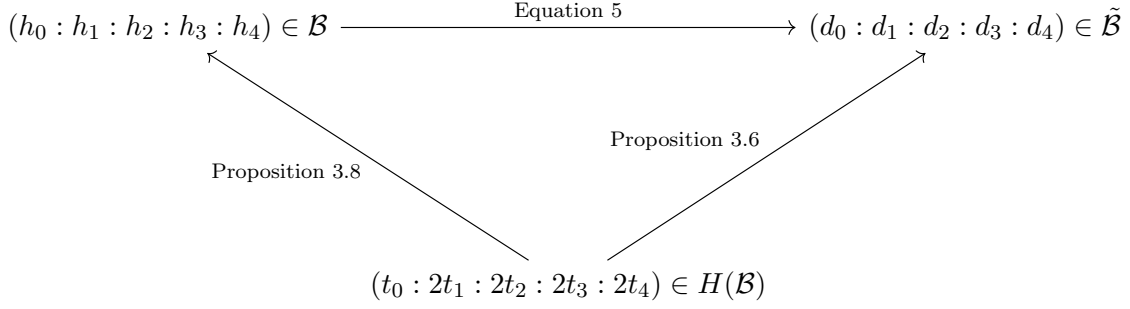
\begin{figure}
    \centering
\begin{tikzcd}[row sep=huge, column sep=tiny]
(h_0:h_1:h_2:h_3:h_4)\in \B \arrow[rr,"\text{Equation }\ref{eq:gradient-embedding}"] & & (d_0:d_1:d_2:d_3:d_4)\in \tilde\B \\ \\
& (t_0 : 2t_1 : 2t_2 : 2t_3 : 2t_4) \in H(\B) \arrow[uul,"\text{Proposition }\ref{prop:t-to-h}"] \arrow[uur,"\text{Proposition }\ref{prop:t-to-d}"]
\end{tikzcd}
    \caption{Relations between the used elements of the Burkhardt quartic, its dual and its Hessian.}
    \label{fig:relations-coefficients}
\end{figure}

\begin{proposition} \label{prop:t-to-d}
	Let $0_{\A_{d,h}}$ be the neutral element with coordinates defined by $t_0, \dots, t_4$ as in Equation \ref{eq:neutral-element}. If $t_0^3 + 2(t_1^3 + t_2^3 + t_3^3 + t_4^3) \neq 0$, then $d = (d_0: \dots: d_4)$ with 
	\[
	d_0 = t_0^3 + 2(t_1^3 + t_2^3 + t_3^3 + t_4^3), \quad d_i = 3t_0t_i^2 + 6 \prod_{j\neq 0, i} t_j.
	\]
\end{proposition}

\begin{proof}
	The values for $d_i$ are found by evaluating the cubic equations $F_1, \dots, F_4$ from Theorem \ref{thm:defining-equations-dim2} in the neutral element. This yields a unique solution if the sum $t_0^3 + 2(t_1^3 + t_2^3 + t_3^3 + t_4^3)$ does not vanish. 
\end{proof}

\begin{corollary} \label{cor:non-zero-coordinates}
	Let $\A_{d,h}$ be an abelian surface in Hesse form, and  $0_{\A_{d,h}} = (t_0 : t_1 : t_1 : t_2 : t_3: t_4 : t_2 : t_4 : t_3)$ the neutral element. Then there are at least two indices $i \neq j$ with $t_i,t_j \neq 0$. 
\end{corollary}

\begin{proof}
	Assume that there is only one index $i$ with $t_i \neq 0$. Then by Proposition \ref{prop:t-to-d}, we have $d = (1 : 0 : 0 : 0: 0)$. However, this point defines a singular surface. Instead of providing an explicit computation, we note that this can be deduced from the moduli interpretation: the point $d = (1:0:0:0:0) \in \tilde{B}$ belongs to a cusp \cite[Section 5]{freitag2004burkhardt1}.
\end{proof}

The bound in the corollary is sharp, in the sense that there exist $\A_{d,h}$ with exactly two $t_i, t_j \neq 0$ for $i\neq j$ (see Subsection \ref{subsec:nongeneric}). 

Furthermore, we recall that $(t_0:2t_1: \dots :2t_4)$ is a point on $H(\tilde{B})$ (cf. Equation \ref{eq:hessian-relation}), hence after scaling the coordinates appropriately, the correspondence from Proposition \ref{prop:t-to-d}  defines a map 
\[
H(\B)\cap \{4 t_0^3 + t_1^3 + t_2^3 + t_3^3 + t_4^3 \neq 0\} \to \tilde{\B}, \quad (t_0: \dots: t_4) \mapsto (d_0: \dots: d_4).
\] Interestingly, the description of this map coincides with the gradient operator $\nabla: \B \to \tilde{\B}$ from Equation \ref{eq:gradient-embedding}. 

\begin{proposition} \label{prop:t-to-h}
	Let $0_{\A_{d,h}}$ be the neutral element with coordinates defined by $t_0, \dots, t_4$ as in Equation \ref{eq:neutral-element}. Then $h= (h_0: \dots: h_4)$ is an element in the kernel of the matrix
    \[
    \text{Hess}(F)(t_0,2t_1,2t_2,2t_3,2t_4) = \left( \frac{\partial F}{\partial x_i \partial x_j}\right)(t_0,2t_1,2t_2,2t_3,2t_4),
    %\begin{pmatrix}
    %    t_{0}^{2} & t_{1}^{2} & t_{2}^{2} & t_{3}^{2} & t_{4}^{2} \\
    %    t_{1}^{2} & t_{0} t_{1} & t_{3} t_{4} & t_{2} t_{4} & t_{2} t_{3} \\
    %    t_{2}^{2} & t_{3} t_{4} & t_{0} t_{2} & t_{1} t_{4} & t_{1} t_{3} \\
    %    t_{3}^{2} & t_{2} t_{4} & t_{1} t_{4} & t_{0} t_{3} & t_{1} t_{2} \\
    %    t_{4}^{2} & t_{2} t_{3} & t_{1} t_{3} & t_{1} t_{2} & t_{0} t_{4}  
    %\end{pmatrix}.
    \]
    where $\text{Hess}(F)$ denotes the Hessian matrix of the equation $F$ as in Definition \ref{def:hessian}.
    Generically, there is a unique nontrivial solution. Moreover, if $t_0t_i^2 - \prod_{j\neq 0, i}t_j\neq 0$ for all $i\in\{1,2,3,4\}$, then this solution is given by $h = (1 : h_1:h_2:h_3:h_4)$ with
	\[
	h_i = -\frac{t_0^3 + 2t_i^3 - \sum_{j \neq 0,i}t_j^3}{3(t_0t_i^2 - \prod_{j\neq 0, i}t_j)}.
	\]
\end{proposition}

\begin{proof}
    The first part of the statement is classical. It underlies the proof of the fact that the Burkhardt quartic is self-Steinerian \cite[Section 5.3.1]{hunt2006burkhardt}. We repeat the argument here, since this will be helpful for the second part of the proof.
    Evaluating the nine quadratic equations $G_0,\ldots,G_8$ from Theorem \ref{thm:defining-equations-dim2} in the neutral element $0_{\A_{d,h}}$ yields five distinct linear conditions $c_0, \dots, c_4$  that need to be satisfied:
    \begin{align*}
    (c_0): && 0 = &~ t_{0}^{2} h_{0} + t_{1}^{2} h_{1} + t_{2}^{2} h_{2} + t_{3}^{2} h_{3} + t_{4}^{2} h_{4}, \\
    (c_1): && 0 = &~t_{1}^{2} h_{0} + t_{0} t_{1} h_{1} + t_{3} t_{4} h_{2} + t_{2} t_{4} h_{3} + t_{2} t_{3} h_{4},\\
    (c_2): && 0 = &~ t_{2}^{2} h_{0} + t_{3} t_{4} h_{1} + t_{0} t_{2} h_{2} + t_{1} t_{4} h_{3} + t_{1} t_{3} h_{4}, \\  
    (c_3): && 0= &~ t_{3}^{2} h_{0} + t_{2} t_{4} h_{1} + t_{1} t_{4} h_{2} + t_{0} t_{3} h_{3} + t_{1} t_{2} h_{4}, \\  
    (c_4): && 0 = &~t_{4}^{2} h_{0} + t_{2} t_{3} h_{1} + t_{1} t_{3} h_{2} + t_{1} t_{2} h_{3} + t_{0} t_{4} h_{4}.
    \end{align*}
    These conditions translate to $h$ being an element in the kernel of some $5 \times 5$ matrix. One can see that this matrix is exactly the Hessian of $F$ evaluated in the point $(t_0:2t_1:2t_2:2t_3:2t_4)$. Hence, there exists a nontrivial element in the kernel, and generically this element is unique (as a projective point). 

    The explicit formulas in the statement can be deduced from the conditions $c_0, \dots, c_4$ by considering the cubic equations
    \begin{align*}
        (c_0t_0 + 2c_1t_1 - c_2t_2-c_3t_3-c_4t_4): && 0 =& h_0 (t_0^3 + 2 t_1^3) + 3 h_1 (t_0t_1^2 - t_2t_3t_4),\\
        (c_0t_0 + 2c_2t_2 - c_1t_1-c_3t_3-c_4t_4): && 0 =& h_0 (t_0^3 + 2 t_2^3) + 3 h_2 (t_0t_2^2 - t_1t_3t_4),\\
        (c_0t_0 + 2c_3t_3 - c_1t_1-c_2t_2-c_4t_4): && 0 =& h_0 (t_0^3 + 2 t_3^3) + 3 h_3 (t_0t_3^2 - t_1t_2t_4),\\
        (c_0t_0 + 2c_4t_4 - c_1t_1-c_2t_2-c_3t_3): && 0 =& h_0 (t_0^3 + 2 t_4^3) + 3 h_4 (t_0t_4^2 - t_1t_2t_3).
    \end{align*}
    We note that these correspond to the cubic relations of type $(d)$ in \cite[Theorem 4.1]{gunji2006defining} (see also \cite[Section 8]{gunji2006defining}).
\end{proof}

Proposition \ref{prop:t-to-h} defines a map $H(\B) \to \B$. We point out that this map is equal to the so-called {\em Steinerian map} from $H(\B)$ to the {\em Steinerian of $\B$} which (remarkbly so) is equal to $\B$ \cite[Section B.1.1.3]{hunt2006burkhardt}.

\subsection{Symplectic transformations} \label{subsec:symplectic-trafos}
	
	The coefficient vectors $(d,h)$ of a p.p. abelian surface in Hesse form $\A_{d,h}$ are not isomorphism invariants. Coefficient vectors defining isomorphic p.p. abelian surfaces are related by symplectic transformations of the underlying theta structure. These correspond to automorphisms of the Burkhardt quartic. 
	Here, we explain how the automorphisms of the Burkhardt quartic act on $\A_{d,h}$.
	
	\begin{notation} \label{not:scaling}
		For $\lambda = (\lambda_0: \dots : \lambda_4) \in \PP^4$, we denote by $\star$ the coordinate-wise scaling
		\[
		\lambda \star x = (\lambda_0 x_0: \dots:  \lambda_4 x_4) \quad \text{for } x \in \PP^4,
		\]
		and when all $\lambda_i \neq 0$, we write 
		\[
		\lambda^{-1} = (\lambda_0^{-1} : \dots : \lambda_4^{-1})
		\]
		for the coordinate-wise inverse. 
		
		We further extend the scaling map to $ x \in \A_{d,h} \subset \PP^8$:
		\begin{equation*}
			C_\lambda(x) = (\lambda_0 x_0 : \lambda_1 x_1 : \lambda_1 x_2 : \lambda_2 x_3 : \lambda_3 x_4 : \lambda_4 x_5: \lambda_2 x_6 : \lambda_4 x_7 : \lambda_3 x_8).
		\end{equation*}
		The scaling is symmetric in the sense that it commutes with multiplication by $[-1]$.
	\end{notation}
	
	Note that the image $C_\lambda(\A_{d,h})$ is not necessarily in Hesse form. In the next proposition, we provide values $\lambda \in \PP^4$ for which this is the case.
	
	\begin{proposition}\label{prop:symplectic-transformation-1}
		Let $\A_{d,h} \subset \PP^8$ be a p.p. abelian surface in Hesse form, and consider the scaling vectors
		\[
		a_0 = (1 : \omega : 1 : \omega : \omega), \quad 
		a_1 = (1 : 1 : 1 : \omega^2 : \omega), \quad 
		a_2 = (1 : 1 : \omega : \omega : \omega).
		\]
		Then for any $\lambda \in \{t_0,t_1,t_2\} $, the image $C_\lambda(\A_{d,h})$ is again in Hesse form, more precisely
		\[
		C_{\lambda}(\A_{d,h}) = \A_{\lambda^{-1} \star d, \lambda \star h}.
		\]
		%Furthermore, the identity element of $C_{\lambda}(\A_{d,h})$ is given by 
		%$0_{C_\lambda(\A_{d,h})} = C_\lambda (0_{\A_{d,h}})$.
		
		%\begin{align*} 
		%   s_0: (x_0: \dots: x_8) &\mapsto (\omega^2 x_0:  x_1 : x_2 : \omega^2 x_3 :  x_4 :  x_5 : \omega^2 x_6 :  x_7 :  x_8), \\
		%  s_1: (x_0 : \dots : x_8) & \mapsto (x_0: x_1 : x_2 : x_3 : \omega x_4 : \omega^2 x_5 : x_6 : \omega^2 x_7 : \omega x_8), \\
		%  s_2: (x_0 : \dots : x_8) & \mapsto (\omega^2 x_0: \omega^2 x_1 : \omega^2 x_2 : x_3 : x_4 : x_5 : x_6 : x_7 : x_8).\\
		%\end{align*}
		
		%These transformations preserve the Hesse form. More precisely:
		%\begin{itemize}
		%   \item $s_0(\A_{d,h}) = \A_{d,h'}$ with $d' = (\omega^2 d_0: d_1 : \omega^2 d_2 : d_3 : d_4)$, $h' = (h_0 : \omega h_1 : h_2 : \omega h_3 : \omega h_4)$. 
		%  \item $s_1(\A_{d,h}) = \A_{d,h'}$ with  $h' = (h_0: h_1 : h_2 : \omega h_3: \omega^2 h_4)$.
		% \item $s_2(\A_{d,h}) = \A_{d,h'}$ with  $h' = (h_0 : h_1 : \omega h_2 : \omega h_3 : \omega h_4)$. 
		%\end{itemize}
	\end{proposition}
	
	\begin{proof}
		This can be verified by an explicit computation, using the description of $\A_{d,h}$ provided in Theorem \ref{thm:defining-equations-dim2}.
	\end{proof}
	
	We note that the transformations described in Proposition \ref{prop:symplectic-transformation-1} are (necessarily) automorphisms of the Burkhardt quartic. The action of $C_{\lambda_i}$ on $h$ corresponds to the element $S_i$ from Eq. \ref{eq:automorphisms} and as discussed in Subsection \ref{subsec:burkhardt-dual}, this extends to an action on the dual, i.e. the action on $d$.

	The automorphism described by the matrix $M$ (cf. Eq. \ref{eq:automorphisms}) can be interpreted as a discrete Fourier transform on $\A_{d,h}$. 
	
	\begin{definition} \label{def:DFT}
		Let 
		\begin{align*}
			 M_{3,3} =  \begin{pmatrix}
				M_3 & M_3 & M_3\\
				M_3 & \omega M_3 & \omega^2 M_3\\
				M_3 & \omega^2 M_3 & \omega M_3
			\end{pmatrix}, \quad \text{with }
			M_3 =  \begin{pmatrix}
				1 & 1 & 1\\
				1 & \omega & \omega^2 \\
				1 & \omega^2  & \omega 
			\end{pmatrix}.
		\end{align*}
		The {\em discrete Fourier transform} is defined as the left action of $M_{3,3}$ on the coordinate vector of a point $P =(x_0:\dots:x_8) \in \A_{d,h}$. We denote
		\[
		M_{3,3}: (x_0:\dots:x_8) \mapsto {M_{3,3}} \cdot (x_0: \dots: x_8).
		\]  
\end{definition}
	
\begin{proposition}\label{prop:dual-coordinates-dim2}
		Let $\A_{d,h}$ be a p.p. abelian surface in Hesse form. Then the discrete Fourier transform  defines an isomorphism $\A_{d, h} \to \A_{\tilde d,\tilde h}$ with
		\[
		\tilde d = M\cdot d, \quad {\tilde h} = h \cdot M,\quad \tilde t = M \cdot t.
		\]
\end{proposition}
	
\begin{proof}
	This can be verified by an explicit computation, using the description of $\A_{d,h}$ provided in Theorem \ref{thm:defining-equations-dim2}.	
\end{proof}
	
Similarly as in Proposition \ref{prop:symplectic-transformation-1}, the action of the discrete Fourier transform on $h$ (resp. $d$) corresponds to the action of $M$ from Eq. \ref{eq:automorphisms} on the Burkhardt quartic (resp. the dual of the Burkhardt quartic).

    \subsection{Product surfaces}\label{subsec:product}

The description of abelian surfaces in Hesse form in Theorem \ref{thm:defining-equations-dim2} includes product surfaces. Here, we study this special case in more detail. We recall that an elliptic curve in Hesse form is defined by a cubic equation
\begin{equation*}
    \H_d: d_1(X^3 + Y^3 + Z^3) = 3d_0 XYZ.
\end{equation*} for some $d = (d_0:d_1) \in \PP^1 \setminus\{1,\omega,\omega^2,\infty\}$ (see Eq. \ref{eq:ell-hesse}), and set the neutral element to $0_{\H_d} = (0:-1:1)$.

\begin{lemma} \label{lem:product-equation}
    Let $\H_{(d_0:d_1)}$, $\H_{(d_0':d_1')}$ be two elliptic curves in Hesse form. Further, let $A = \H_{(d_0:d_1)} \times \H_{(d_0':d_1')}$ and consider the Segre embedding 
    \begin{align*}
    \iota_S: A &\to \PP^8, \\
    \left((x:y:z), (x':y':z') \right) &\mapsto (xx':xy':xz':yx':yy':yz':zx':zy':zz').
    \end{align*}
    Then 
    \(
    \iota_S(A) = \A_{d,h}
    \), with
    \[d = (d_0d_0': d_0d_1': d_1d_0':d_1d_1':d_1d_1'),
    \quad 
    h = (0:0:0:-1:1).
    \]
    The neutral element is
    \[
    0_{\A_d} =  (0:0:0:0:1:-1:0:-1:1).
    \]
\end{lemma}

\begin{proof}
   We denote
    \[
    (x_0: \dots: x_8) = (xx':xy':xz':yx':yy':yz':zx':zy':zz')
    \]
    for the generic point on $\iota_S(A)$. It suffices to show that this point satisfies the equations from Theorem \ref{thm:defining-equations-dim2}.

    The quadratic relations are induced by the Segre embedding. More precisely, the image of $\iota_S(\PP^2 \times \PP^2) \subset \PP^8$ is defined by the zero set of 
    \[
    x_{3i+j}x_{3k+\ell} - x_{3i_\ell}x_{3k+j} \quad \text{for } i,j,k,\ell \in \{0,1,2\}. 
    \]
    These relations yield the $8$ quadrics from Theorem  \ref{thm:defining-equations-dim2} with coefficient vector $h = (0:0:0:-1:1)$.

    The cubic relations are induced by the equations for $\H_{(d_0:d_1)}$ and $\H_{(d_0':d_1')}$. To see this, we note that 
    \begin{align*}
    x_0^3 \quad + \quad \dots \quad+ \quad  x_8^3\quad =&~ (x^3 + y^3 + z^3) \cdot (x'^3 + y'^3 + z'^3),\\
    x_0x_1x_2 + x_3x_4x_5 + x_6x_7x_8= &~ (x^3+y^3+z^3) \cdot (x'y'z'),\\
    x_0x_3x_6 + x_1x_4 x_7 + x_2x_5x_8 = &~ (x'^3+y'^3+z'^3) \cdot (xyz),\\
    x_0x_4x_8 + x_1x_5x_6 + x_2x_3x_7 = &~ 3 \cdot (xyz) \cdot (x'y'z'),\\
    x_0x_5x_7 + x_1x_3x_8 + x_2x_4x_6 = &~ 3 \cdot (xyz) \cdot (x'y'z').\\
    \end{align*}
    Substituting these identities into the cubic relations from Theorem  \ref{thm:defining-equations-dim2} with coefficient vector 
    \[
    d = (d_0d_0': d_0d_1': d_1d_0':d_1d_1':d_1d_1')
    \]
    reduces to the defining equations of the elliptic curves.

    Finally, the neutral element is computed as $\iota_S\left((0:-1:1),(0:-1:1)\right)$
\end{proof}

Note that the neutral element is as in Equation \ref{eq:neutral-element} with $t_0 = t_1 = t_2 = 0$ and $t_3 = - t_4 = 1$. In particular it is an {\em even} $2$-torsion point.

\begin{comment}
\begin{lemma}[$3$-torsion on the product surface] 
\label{lem:product-torsion} 
    Let $A = H_d \times H_{d'}$ and $\iota_S: A\to \PP^8$ be the Segre embedding as in Lemma \ref{lem:product-equation}. 
    Then $\mathcal{B}_A = (P_1,P_2,Q_1,Q_2)$ with 
    \begin{align*}
    \iota_S(P_1) =& (0:0:0:1:-1:0:-1:1:0),\\
    \iota_S(P_2) =& (0:1:-1:0:-1:1:0:0:0),\\
    \iota_S(Q_1) =&  (0:0:0:0:1:-\omega:0:-1:\omega),\\
    \iota_S(Q_2) =& (0:0:0:0:1:-1: 0: -\omega: \omega), \\
    \end{align*}
    is a symplectic basis of the $3$-torsion. 
\end{lemma}

\begin{proof}

    Let $\mathcal{B} = (P,Q)$ with $P = (-1:1:0)$ , $Q = (0:1:-\omega)$ be a (necessarily symplectic) basis for $H_d[3]$ and similarly $\mathcal{B}' = (P',Q')$ with $P' = (-1:1:0)$, $Q' = (0:1:-\omega)$ a basis for $H_{d'}[3]$. Then a symplectic basis for $A = H_d \times H_d'$ is given by
    \[
    \mathcal{B}_A = \left(
    (0,P'), (P,0), (0,Q'), (Q,0)
    \right).
    \]
    Applying the Segre embedding, yields the coordinates stated in the lemma. 

    Alternatively, the basis may be computed as a corollary of Proposition \ref{prop:3-torsion}.
\end{proof}
\end{comment}

If a reducible abelian surface comes from the Segre embedding of two elliptic curves in Hesse form, then we say that it is in product form. More precisely:

\begin{definition} \label{def:product-structure}
    We say that a reducible principally polarized abelian surface in Hesse form $\A_{d,h}$ is equipped with the product polarization if and only if $h = (0:0:0:-1:1)$.
\end{definition}

Note that in general, for a reducible abelian surface, the quadratic equations in Theorem \ref{thm:3-isogeny-dim2} are defined by some coefficient vector $h$ which is a singular point on the Burkhardt quartic. In total, there are $45$ singular points. As explained in Lemma \ref{lem:singularities-burkhardt}, the automorphism group of the Burkhardt quartic acts transitively on these singular points. Together with our analysis on the action of symplectic transformations in Subsection \ref{subsec:symplectic-trafos}, this implies the following fact.

\begin{corollary} \label{cor:trafo-t-product}
    Let $\A_{d,h}$ be a {\em reducible} principally polarized abelian surface in Hesse form. Then there exists a symplectic transformation $T$ so that $T(\A_{d,h}) = \A_{d',h'}$ is equipped with the product polarization.
\end{corollary}

Note that Corollary \ref{cor:trafo-t-product} is effective. In our implementation,  we describe a concrete symplectic transformation into product form for each singularity of the Burkhardt quartic.
    \section{The $3$-torsion group} \label{sec:three-torsion}

Let $\A_{d,h}$ be a principally polarized abelian surface in Hesse form. 
Recall that the translation by  $3$-torsion points act as linear transformations on the coordinates of a point.  This means that the group $G = (\ZZ/3\ZZ)^2 \times (\ZZ/3\ZZ)^2$ acts on $\A_{d,h} \subset \PP^8$. This action is normalized, and we fix generators
\[
G \cong \langle \sigma_1, \sigma_2, \tau_1, \tau_2 \rangle,
\] with action on a point $P = (x_0: \dots : x_8) \in \A_{d,h}$ given by
\begin{align}
	\label{eq:action-by-3-torsion}
	\begin{split}
		\sigma_1(P) &= (x_1 : x_2: x_0 : x_4 : x_5 : x_3 : x_7 : x_8 : x_6),\\
		\sigma_2(P) &= (x_3 : x_4: x_5 : x_6 : x_7 : x_8 : x_0 : x_1 : x_2),\\
		\tau_1(P) & = (x_0 : \omega x_1 : \omega^2x_2 : x_3 : \omega x_4 : \omega^2 x_5 : x_6 : \omega x_7 : \omega^2 x_8),\\
		\tau_2(P) & = (x_0 : x_1 : x_2 : \omega x_3 : \omega x_4 : \omega x_5 : \omega^2x_6 : \omega^2x_7 : \omega^2x_8),
	\end{split}
\end{align}
That is  $\sigma_1$, $\sigma_2$ act as permutations; and $\tau_1$, $\tau_2$ act by scaling with third roots of unity. 

\begin{proposition} \label{prop:3-torsion}
	Let $\A_{d,h}$ be a principally polarized abelian surface in Hesse form, and let
    \[
    0_{\A_{d,h}} = (t_0 : t_1 : t_1 : t_2 : t_3: t_4 : t_2 : t_4 : t_3),
    \] be the neutral element. Then $(P_1,P_2,Q_1,Q_2)$ with
    \[
    P_1 = \sigma_1(0_{\A_{d,h}}),\quad P_2 = \sigma_2(0_{\A_{d,h}}), \quad Q_1 = \tau_1(0_{\A_{d,h}}),\quad Q_2 = \tau_2(0_{\A_{d,h}})
    \]
	%\begin{align*}
	%	P_1  =& (t_1 : t_1 : t_0:  t_3: t_4 : t_2 :  t_4 : t_3 : t_2)\\
	%	P_2  =& (t_2 : t_3 : t_4 : t_2 : t_4 : t_3 : t_0 : t_1 : t_1 )\\
	%	Q_1  =& (t_0 : \omega t_1 : \omega^2t_1 : t_2 : \omega t_3: \omega^2t_4 : t_2 : \omega t_4 : \omega^2 t_3)\\
	%	Q_2  =& (t_0 : t_1 : t_1 : \omega t_2 : \omega t_3: \omega t_4 : \omega^2 t_2 : \omega^2 t_4 : \omega^2 t_3)
	%\end{align*}
	is a symplectic basis of $\A_{d,h}[3]$. 
\end{proposition}

\begin{proof}
	By the definition of $\sigma_1,\sigma_2,\tau_1,\tau_2$, the points $P_1,P_2,Q_1,Q_2$ form a basis of  $\A_{d,h}[3]$.  To see that it is  symplectic, we evaluate the commutator pairing.\footnote{A formal definition of the pairing is given in \cite[p.227]{mumford-book} and it is shown to coincide with the Weil pairing, see also \cite[Section 6.1]{lubicz2012computing}.} For this purpose, we denote by $S_1,S_2,T_1,T_2$ the $9\times 9$ matrices which describe the action of $\sigma_1,\sigma_2,\tau_1,\tau_2$ on the coordinate vectors. We compute
	\[
	e_3(P_i,P_j) = \frac{S_iS_j}{S_jS_i} = 1, \quad e_3(Q_i,Q_j) = \frac{T_iT_j}{T_jT_i} = 1
	\]
	and 
	\[
	e_3(P_i, Q_j) = \frac{S_iT_j}{T_jS_i} = \begin{cases}
		\omega & \text{if } i = j,\\
		1 & \text{if } i\neq j.
	\end{cases}
	\]
\end{proof}

\begin{definition} \label{def:symplectic-decomposition}
	We call the basis $(P_1,P_2,Q_1,Q_2)$ the {\em canonical basis} of $\A_{d,h}[3]$, and define the {\em symplectic decomposition} $\A_{d,h}[3] = K_1 + K_2$ with $K_1 = \langle P_1, P_2 \rangle$ and $K_2 = \langle Q_1,Q_2 \rangle$.
\end{definition}
Note that $K_1$ and $K_2$ are maximal isotropic subgroups of $\A_{d,h}[3]$.

For product surfaces, the canonical symplectic basis is induced by the normalized bases of the elliptic curves in Hesse form. We make this precise in the following statement.
\begin{corollary} \label{cor:product-basis}
    Let $\H_d$ and $\H_d'$ be elliptic curves in Hesse form, and consider the Segre embedding
    \[
    \iota_S : \H_d \times \H_{d'} \to \A_{d,h} \subset \PP^8 
    \]
    as in Lemma \ref{lem:product-equation}. Further let $(P,Q)$ and $(P',Q')$ be the canonical bases of $\H_{d}[3]$ and $\H_{d'}[3]$, respectively, cf. Eq. \ref{eq:ell-hesse-torsion}. That is
    \[
    P = (-1:1:0), Q = (0:-\omega^2:1) \in \H_d,\quad P' = (-1:1:0), Q' = (0:-\omega^2:1) \in \H_{d'}.
    \]
    Then the canonical basis of $\A_{d,h}[3]$ is given by $(P_1,P_2,Q_1,Q_2)$ with
    \[
    P_1 = \iota_S(0, P'), \quad P_2 = \iota_S(P, 0), \quad Q_1 = \iota_S(0,Q'), \quad Q_2 = \iota_S(Q,0).
    \]
\end{corollary}

\subsection{Action of symplectic transformations on the $3$-torsion}
For a given $(3,3)$-subgroup $K$ of $\A_{d,h}$, one can consider the isogeny $\phi: \A_{d,h} \to \A_{d',h'}$ with kernel $K$. In Section \ref{sec:isogeny}, we provide explicit formulas for computing such an isogeny, but it requires that the input is a specific $(3,3)$-subgroup:  the group $K_2$ from the canonical symplectic decomposition (Definition \ref{def:symplectic-decomposition}). In order to transfer an arbitrary $(3,3)$-group $K$ to a group of the correct form, it is necessary to apply a suitable symplectic transformation. For this purpose, we study the action of symplectic transformations described in \ref{subsec:symplectic-trafos} on the $3$-torsion.

\begin{lemma} \label{lem:DFT-3torsion}
	Let $\A_{d,h}$ be a principally polarized abelian surface in Hesse form, and consider the discrete Fourier transform $\DFT: \A_{d,h} \to \A_{d',h'}$. Further, let $(P_1,P_2,Q_1,Q_2)$ and $(P_1', P_2', Q_1', Q_2')$ be the canonical bases of $\A_{d,h}$ and $\A_{d',h'}$, respectively. Then
	\[
	\DFT(P_1) = - Q_1', \quad  \DFT(P_2) = - Q_2', \quad \DFT(Q_1) = P_1', \quad \DFT(Q_2) =  P_2'.
	\]
	In particular $\DFT(K_1) = K_2'$ and $\DFT(K_2) =  K_1'$.
\end{lemma}

\begin{proof}
	This is verified by a direct computation.
\end{proof}

\begin{lemma} \label{lem:symplectic-3torsion}
	Let  $\A_{d,h}$ be a principally polarized abelian surface in Hesse form, and consider the symplectic transformations 
	\[
	C_{a_i}: \A_{d,h} \to C_{a_i}(\A_{d,h}), \quad \text{for } i = 0,1,2,
	\]
	from Proposition \ref{prop:symplectic-transformation-1}. 
	%Further, let $(P_1,P_2,Q_1,Q_2)$ and $(P_1^i,P_2^i,Q_1^i,Q_2^i)$ be the canonical bases of $\A_{d,h}$ and $C_{t_i}(\A_{d,h})$, respectively.
	Then these transformations map the canonical basis of $\A_{d,h}$ to the canonical basis of $C_{a_i}(\A_{d,h})$ as respectively
	\[
	\begin{pmatrix}
		1 & 0 & 1 & 0\\
		0 & 1 & 0 & 0\\
		0 & 0 & 1 & 0\\
		0 & 0 & 0 & 1
	\end{pmatrix},\quad 
	\begin{pmatrix}
		1 & 0 & 0 & 1\\
		0 & 1 & 1 & 0\\
		0 & 0 & 1 & 0\\
		0 & 0 & 0 & 1
	\end{pmatrix},\quad 
	\begin{pmatrix}
		1 & 0 & 0 & 0\\
		0 & 1 & 0 & 1\\
		0 & 0 & 1 & 0\\
		0 & 0 & 0 & 1
	\end{pmatrix}.
	\]
\end{lemma}
\begin{proof}
	This is verfied by a direct computation; see also \cite[Proposition 1]{freitag2004burkhardt1}. 

\end{proof}

    \section{Abelian surfaces in twisted Hesse form} \label{sec:twisted-surface}

This section is a brief interlude describing twists of abelian surfaces in Hesse form. These will appear in the proof of our main theorem in Section \ref{thm:3-isogeny-dim2}. While we do not explore the construction in full detail, we remark that this can be viewed as the two-dimensional analogue of elliptic curves in {\em twisted Hesse form} introduced in \cite{bernstein2015twisted}.

\begin{definition} \label{def:twisted-surface}
    Let $\A_{d,h}$ be a principally polarized abelian surface in Hesse form, and let $\lambda = (\lambda_0: \dots : \lambda_4) \in \PP^4$. We call $C_\lambda(\A_{d,h})$ the {\em twist of $\A_{d,h}$ by $\lambda$}. 
\end{definition}

Note that twisting a surface in Hesse form does not affect the action of the $3$-torsion subgroup $K_2$. For future reference, we state this observation in the next lemma. 

\begin{lemma} \label{lem:twisted-3-torsion-action}
   Let $A$ be a p.p. abelian surface which is the twist of some p.p. abelian surface $\A_{d,h}$ in Hesse form, i.e. there exists $\lambda \in \PP^4_{\bar{k}}$ so that $A = C_\lambda(\A_{d,h})$. 

   If the base field $k$ contains a primitive root of unity $\omega$, then there are rational $3$-torsion points $Q_1, Q_2 \in A[3]$, and their action on points in $A$ is given by
   \[
    P + Q_1 = \tau_1(P), \quad P + Q_2 = \tau_2(P), \quad \text{for all } P \in A
   \]
   with $\tau_1,\tau_2$ as in Eq. \ref{eq:action-by-3-torsion}.
   Moreover multiplication by $-1$ is given by 
    \[
    \begin{tikzcd} [row sep = 15pt] 
    ~P = (x_0:x_1:x_2:x_3:x_4:x_5:x_6:x_7:x_8) \dar[maps to]\\
    -P = (x_0:x_2:x_1:x_6:x_8:x_7:x_3:x_5:x_4)
    \end{tikzcd}
    \]
    for all $P \in A$.
\end{lemma}

\begin{proof}
    Consider the canonical basis $(\tilde{P}_1, \tilde{P}_2, \tilde{Q}_1, \tilde{Q}_2)$ of $\A_{d,h}[3]$. The map $C_\lambda$ is given by scaling the coordinates of a point. The same is true for the action of the $3$-torsion points $\tilde{Q}_1$ and $\tilde{Q}_2$ (even though in the latter case, the scaling is (necessarily) not symmetric with respect to multiplication by $[-1]$). Consequently the action of these $3$-torsion points  commutes with the application of the map $C_\lambda$. This implies that the points $Q_1 = C_{\lambda}(\tilde{Q}_1)$ and $Q_2 = C_{\lambda}(\tilde{Q}_2)$ in $A = C_\lambda(\A_{d,h})$ have the properties described in the lemma. 

    The translation of the multiplication by $[-1]$ map  via $C_\lambda$ is straightforward.
\end{proof}

We note that it is in general not true that the abelian surface from Lemma \ref{lem:twisted-3-torsion-action} has full rational $3$-torsion. This is only the case if the twist is defined by a $k$-rational element $\lambda \in \PP^4$. 
    
\section{$(3,3)$-isogenies} \label{sec:isogeny}

Let $\A_{d,h}$ be a principally polarized abelian surface in Hesse form. We consider the symplectic decomposition $\A_{d,h}[3] = K_1 + K_2$ (cf. Definition \ref{def:symplectic-decomposition}) with maximal $3$-isotropic groups $K_1,K_2$.  

\begin{notation} \label{not:basic-operations}
    We recall our notation for the following basic operations on points $x = (x_0: \dots : x_8) \in \A_{d,h}$.
    \begin{itemize}
    \item $\cube$:  coordinate-wise cubing, i.e. $\cube(x) = (x_0^3 : \dots : x_8^3)$
    \item $M_{3,3}$: The discrete Fourier transform (cf. Definition \ref{def:DFT}).
    \item $C_{\lambda}$ with $\lambda = (\lambda_0: \dots : \lambda_4)$: coordinate-wise scaling (cf. Notation \ref{not:scaling}).
    \end{itemize}
\end{notation}

Our main result states that a $3$-isogeny of abelian surfaces in Hesse form may be computed as the composition of the basic operations from Notation \ref{not:basic-operations}, and provides an explicit formula for the scaling vector $\lambda \in \PP^4$ in terms of certain $9$-torsion points lying above the kernel.

\begin{lemma} \label{lem:3-isogeny-form}
	Let $\A_{d,h}$ be a p.p. abelian surface in Hesse form, and denote $(P_1,P_2,Q_1,Q_2)$ for the canonical $3$-torsion basis.
	Then there exists a scaling vector $\lambda \in \PP^4$ so that
	\[
	\phi: \A_{d,h} \to \A_{d',h'}, \quad P \mapsto C_\lambda \circ M \circ \cube(P)
	\]
	defines a $(3,3)$-isogeny with kernel $K_2 = \langle Q_1, Q_2\rangle$ and the codomain is again in Hesse form. 
\end{lemma}

\begin{proof} 	
	Throughout the proof, we use the following notation for the intermediate abelian surfaces and maps appearing in the computation of $\phi: \A_{d,h} \to \A_{d',h'}$.
	
	\[
	\begin{tikzcd}
		\A_{d,h} \rar{\cube} \arrow[bend right = 70, "\phi_1"]{r}  \arrow[bend right = 70, "\phi_2"]{rr}  & A_1 \rar{M_{3,3}} & A_2 \rar{C_\lambda} & \A_{d',h'}.
	\end{tikzcd}
	\]
	
	First, we note that the map $\phi_1: \A_{d,h} \to A_1$ is an isogeny, and its kernel is given by \[
	\ker(\cube) = \{(x_0:\dots:x_8) \mid (x_0^3 : \dots : x_8^3) = (t_0^3: \dots: t_3^3)\}
	\]  
	which is precisely $K_2=\langle Q_1,Q_2 \rangle$. It remains to show that there exists $\lambda \in \PP^4$ so that applying the linear transformations $M_{3,3}$ and $C_\lambda$ converts the codomain $A_1$ into an abelian surface in Hesse form $A_{d',h'}$ for some $d',h'$.

	We first study the image of the group $K_1 = \langle P_1, P_2 \rangle$ under  $\phi_2$ and its action on the points of ${A}_2$. 
%	 The image of the neutral element is given by
%	\[
%	\phi_2(0_{\A_{d,h}}) = (\tilde{t}_0:\tilde{t}_1:\tilde{t}_1:\tilde{t}_2:\tilde{t}_3:\tilde{t}_4:\tilde{t}_2:\tilde{t}_4:\tilde{t}_3), 
%	\]
%	where
%	\[
%	\tilde{t} = \begin{pmatrix} \tilde{t}_0\\ \vdots \\ \tilde{t}_4 \end{pmatrix} = {M} \cdot \begin{pmatrix} t_1^3\\ \vdots \\ t_4^3 \end{pmatrix}
%	\]
%	with ${M}$ as in Eq. \ref{eq:automorphisms}.
%	
	The action of the $3$-torsion points $P_1,P_2$ on points in $\A_{d,h}$ translates to ${A}_2$. Symbolically evaluating the generic point and its translates by $P_1$ and $P_2$, respectively, we find that  
	\[
	P + \phi_2({P}_1) = \tau_1^2(P), \quad P + \phi_2({P}_2) = \tau_2^2(P) \quad \text{ for all } P \in {A_2},
	\]
	where $\tau_1,\tau_2$ are the linear transformations described in Section \ref{sec:three-torsion}.  Further note that the description of the multiplication by $[-1]$ map on $A_2$ coincides with that on $\A_{d,h}$.  These are precisely the properties of twisted abelian surfaces described in Lemma \ref{lem:twisted-3-torsion-action}. However, it remains to show that $A_2$ is indeed a twisted abelian surface.
	
	We know that there exists a linear transformation $F:{A}_2 \to \A_{d',h'}$ for some $d',h'$ which is not necessarily defined over the base field. Denote $(P_1',P_2',Q_1',Q_2')$ for the canonical symplectic basis of $\A_{d',h'}[3]$. Without loss of generality, we may assume that $(F \circ\phi_2)({P}_1) = [-1] Q_1'$ and $(F\circ \phi_2)(\tilde{P}_2) = [-1]Q_2'$. Since the action of $[-1]\phi_2({P}_1)$ and $[-1]\phi_2({P}_2)$ is already in canonical form, this further implies $F(a\phi_2({P}_1) +b  \phi_2({P}_2)) = -aQ_1' - bQ_2'$ for all $a,b \in \{0,1,2\}$. In other words, the application of $F$ commutes with the action of $\tau_1$ and $\tau_2$. This implies that $F$ is represented by a diagonal matrix.  Furthermore, $F$ needs to be compatible with multiplication by $[-1]$, hence $F = C_{\lambda}$ for some $\lambda \in \PP^4$.
	
	%Let $P = (x_0:\dots:x_8) \in A_2$ be a generic point. The nine points given by 
	%\[
	%\left(P + a\phi_2(P_1) + b\phi_2(P_2) \right)_{a,b \in \{0,1,2\}} \subset \PP^8
	%\]
	%are in general position. Concretely, they are the columns of $(X \cdot M_{3,3})$, where $X = \diag(x_0,\dots,x_8)$, and $\det(M_{3,3}) = -3^9$. 
	
	%The image of the points under $F$ is of a similar form. Writing $F(P) = P' = (x_0': \dots: x_8')$, and $X' = \diag(x_0' , \dots, x_8')$, we have that 
	%\[
	%F \cdot X \cdot M  = X' \cdot M \qquad \text{in } \mathrm{PGL}_8(\bar{k}).
	%\]
	%It follows that $F$ is represented by a  diagonal matrix,  hence it is given by scaling the coordinates of a point in $A_2$. 
\end{proof}

In the next theorem, we provide an explicit formula for $\lambda$. We show that the correct scaling can be deduced from the knowledge of a $9$-torsion subgroup on the domain.
 First, we provide a formula that only depends on the coordinates of two generators $R_1,R_2$ of this $9$-torsion group (Theorem \ref{thm:3-isogeny-dim2}). In certain exceptional cases, it is necessary to work with $R_1 +R_2$ and $R_1 -R_2$ as additional input (Proposition \ref{prop:isogeny-formula-exceptional}). 

\begin{definition}\label{def:exceptional}
    Let $\A_{d,h}$ be a p.p. abelian surface in Hesse form, and $0_{\A_{d,h}} = (t_0: \dots : t_3) \in \PP^8$ the neutral element, and define the following values 
    \begin{align*}
    \tilde{t}_0 =&~ t_0^3 + 2t_1^3 + 2t_2^3 + 2t_3^3 + 2t_4^3,\\
    \tilde{t}_1 =&~  t_0^3 - t_1^3 + 2t_2^3 - t_3^3 -t_4^3,\\
    \tilde{t}_2 = &~ t_0^3 + 2  t_1^3 - t_2^3 - t_3^3 -t_4^3,\\
    \tilde{t}_3 = &~ t_0^3 - t_1^3 - t_2^3 - t_3^3  +2 t_4^3,\\
    \tilde{t}_4 = &~ t_0^3 - t_1^3 - t_2^3 +2 t_3^3 - t_4^3.
    \end{align*}
    We say that $\A_{d,h}$ is {\em exceptional} if there is an index $i>0$ with $\tilde{t}_i = 0$.
\end{definition}
Note that by definition $\A_{d,h}$ is not necessarily exceptional if $\tilde{t}_0 = 0$. Nevertheless, we included this value in the definition, since it will appear later on.

\begin{theorem} \label{thm:3-isogeny-dim2}
    Let $\A_{d,h}$ be a p.p. abelian surface in Hesse form, and assume that $\A_{d,h}$ is not exceptional. Denote $(P_1,P_2,Q_1,Q_2)$ for the canonical $3$-torsion basis, and let 
    \(
    \langle R_1, \; R_2\rangle \subset \A_{d,h}[9]
    \) be a maximal isotropic subgroup with $3 \cdot R_1  =  Q_1$ and $3 \cdot R_2  =  Q_2$. We denote
    \[
    (x_{i,0}:\dots: x_{i,8}) = M \circ \cube(R_i) \quad \text{for } i =1,2.
    \]
    and set 
    \begin{align*}
    (\lambda_0 : \lambda_1) =&~ (x_{1,1} :  x_{1,0}),
    \\
     (\lambda_0 : \lambda_2) =&~ (x_{2,3} :  x_{2,0}),
     \\
    (\lambda_2 : \lambda_3) =&~ (x_{1,4}: x_{1,6}),
    \\ 
    (\lambda_3 : \lambda_4) =&~ (x_{1,5} : x_{1,8}).  
    \end{align*}
    
    Then 
    \[
    \phi: \A_{d,h} \to \A_{d',h'}, \quad P \mapsto C_\lambda \circ M \circ \cube(P)
    \]
    defines a $(3,3)$-isogeny with kernel $K_2 = \langle Q_1, Q_2\rangle$ such that the codomain is again in Hesse form. Moreover, the canonical symplectic basis for $\A_{d',h'}[3]$ is given by  $(P_1',P_2',Q_1',Q_2')$, where $P_1' = \phi(R_1), \, P_2' = \phi(R_2)$ and $Q_1' = [-1]\phi(P_1), \, Q_2' = [-1]\phi(P_2)$. 
\end{theorem}

\begin{proof} 

Let $C_\lambda$ be as in the proof of Lemma \ref{lem:3-isogeny-form}, and recall the notation 
\[
\begin{tikzcd}
    \A_{d,h} \rar{\cube} \arrow[bend right = 70, "\phi_1"]{r}  \arrow[bend right = 70, "\phi_2"]{rr}  & A_1 \rar{M_{3,3}} & A_2 \rar{C_\lambda} & \A_{d',h'}.
\end{tikzcd}
\]
We show that the coordinates of $\lambda$ are determined by the $9$-torsion points  $R_1$ and $R_2$. 

Since $\langle R_1,R_2\rangle$ is a maximal isotropic subgroup lying above the kernel of the isogeny $\phi$, we have that 
\[
(\phi(R_1),\phi(R_2),\phi(P_1), \phi({P}_2))
\]
is a symplectic basis for $A_2[3]$. Moreover, it follows from the properties of the Weil pairing, and our choice of $9$-torsion points $R_1,R_2$ that
\[
e_3(\phi(R_1), \phi(P_1)) = e_3(P_1,Q_1)^{-1}, \quad  e_3(\phi(R_2), \phi(P_2)) = e_3(P_2,Q_2)^{-1}.
\]
We denote by $(P_1',P_2',Q_1',Q_2')$ the canonical symplectic basis of $\A_{d',h'}$. By construction of $C_\lambda$ (as in the proof of Lemma \ref{lem:3-isogeny-form}), we have that $\phi(P_1) = [-1]Q_1'$ and $\phi(P_2)$ = $[-1]Q_2'$. This implies
\begin{align*}
\phi(R_1) =&~P_1' + a Q_1' + b Q_2', \\
\phi(R_2) =&~P_2' + b Q_1' + c Q_2' 
\end{align*}
for some $a,b,c \in \{0,1,2\}$. It follows from Lemma \ref{lem:symplectic-3torsion} that a change of basis \[(\phi(R_1),\phi(R_2),\phi(P_1), \phi({P}_2)) \mapsto (P_1',P_2',[-1]Q_1',[-1]Q_2')\] is simply given by a scaling, hence it can be absorbed into the map $C_\lambda$, and we may assume that $\phi(R_1) = P_1'$ and $\phi(R_2) = P_2'$. 

Let $0_{\A_{d',h'}} = (t_0': \dots : t_3')$ be the neutral element on the codomain. Note that by the definition of the isogeny,  $t_i' = \lambda_i \cdot \tilde{t}_i$ for all $i$, where $\tilde{t}_i$ is as in Definition \ref{def:exceptional}. In particular $t_i \neq 0$ for all $i = 1, \dots, 4$.

The coordinates of the elements $P_1'$ and $P_2'$ are permutations of the coordinates of the neutral element. In order to recover the scaling vector $\lambda$, we denote
\[
    \phi_2(R_i) = (x_{i,0}:\dots: x_{i,8}) \quad \text{for } i =1,2,
\]
and use the equalities $C_\lambda(\phi_2(R_1)) = P_1'$, i.e.
\[
\left(\lambda_0 x_{1,0}: \lambda_1 x_{1,1} : \dots: \lambda_3 x_{1,8}\right) = 
 (t_1' : t_1' : t_0':  t_3': t_4' : t_2' :  t_4' : t_3' : t_2')
\]
and $C_\lambda(\phi_2(R_2)) = P_2'$, i.e.
\[
\left(\lambda_0 x_{2,0}: \lambda_1 x_{2,1} : \dots: \lambda_3 x_{2,8}\right)  = 
(t_2': t_3' : t_4' : t_2' : t_4' : t_3 ': t_0' : t_1' : t_1' ).
\]
From the first equality, we deduce $(\lambda_0: \lambda_1) = (x_{1,1}:x_{1,0})$, $(\lambda_2:\lambda_3) = (x_{1,4}:x_{1,6})$ and $(\lambda_3:\lambda_4) = (x_{1,5}:x_{1,8})$. From the second one, we deduce $(\lambda_{0}:\lambda_{2}) = (x_{2,3}:x_{2,0})$. 
\end{proof}

\subsection{Multiplication by $3$}

Multiplication by $3$ can be computed as the composition of the $3$-isogeny from Theorem \ref{thm:3-isogeny-dim2} with its dual. The dual of the isogeny in the theorem has kernel $K_2' = \langle Q_1',Q_2'\rangle$, where $K_1' + K_2' = \A_{d',h'}[3]$ is the canonical  decomposition of the $3$-torsion on the codomain. This means that the formulas from Theorem $\ref{thm:3-isogeny-dim2}$ could also be applied to compute the dual isogeny. The only subtlety is the computation of the scaling vector. A direct application of the theorem would require fixing additional $9$-torsion points $S_1,S_2$ lying above $P_1,P_2$ as well. In the following we show that this is not necessary.

\begin{lemma} \label{lem:dual-isogeny}
    Let $\phi: \A_{d,h} \to \A_{d',h'}$ be the $(3,3)$-isogeny from Theorem \ref{thm:3-isogeny-dim2}. Then the dual isogeny $\hat{\phi}: \A_{d',h'} \to \A_{d,h}$ is given by 
    \[
    \hat{\phi}: P \mapsto C_{\lambda'} \circ M_{3,3} \circ \cube (P)
    \]
    with $\lambda' = (t_0/a_0: t_1/a_1: t_2/a_2: t_3/a_3: t_4/a_4)$, where 
    \[
    (a_0:a_1:a_1:a_2:a_3:a_4:a_2:a_4:a_3) = M_{3,3} \circ \cube \circ C_{\lambda} \circ M_{3,3} \circ \cube(0_{\A_{d,h}}).
    \]
\end{lemma}

\begin{proof}
    First note that $\phi(P_1,P_2) = \langle Q_1',Q_2'\rangle$, hence this is the kernel of the dual isogeny $\hat{\phi}$. It follows from Theorem \ref{thm:3-isogeny-dim2} that there exists a vector $\lambda' \in \PP^4$ so that $\hat{\phi}$ is given by $\hat{\phi}: P \mapsto C_{\lambda'} \circ M_{3,3} \circ \cube (P)$. To compute $\lambda'$, it suffices to evaluate the multiplication by $3$ map at the neutral element $0_{\A_{d,h}}$. This provides us with the condition
   \[
   0_{\A_{d,h}} = C_{\lambda'} \circ M_{3,3} \circ \cube \circ C_{\lambda} \circ M_{3,3} \circ \cube(0_{\A_{d,h}})
    \]
    from which we deduce the coordinates of $\lambda'$.
\end{proof}

\subsection{Isogenies with arbitrary kernel}
In Theorem \ref{thm:3-isogeny-dim2}, we provide an explicit formula to compute $(3,3)$-isogenies with kernel given by $K_2 = \langle Q_1, Q_2 \rangle$, where $(P_1,P_2,Q_1,Q_2)$ is the canonical $3$-torsion basis. This formula can be used to compute isogenies with arbitrary kernel by precomposing it with a suitable symplectic transformation (see Subsection \ref{subsec:symplectic-trafos}). Here, we make this precise for kernels of the form $\langle P_1 + aQ_1 + bQ_2, P_2 + bQ_1 + cQ_2\rangle$ which appear naturally in our application. 

\begin{lemma} \label{lem:trafo-kernel} 
    Let $\A_{d,h}$ be a p.p. abelian surface in Hesse form, and $(P_1,P_2,Q_1,Q_2)$ the canonical basis of $\A_{d,h}[3]$. Consider the maximal isotropic subgroup $K = \langle P_1 + aQ_1 + bQ_2, P_2 + bQ_1 + cQ_2\rangle \subset \A_{d,h}[3]$. Then the transformation
    \(
    C_{\lambda(a,b,c)}: \A_{d,h}   \to \A_{d',h'},
    \)
    with
    \[
    \lambda = (1 : \omega^{2a}: \omega^{2c}: \omega^{2a+b+2c}: \omega^{2a+2b+2c})
    \]
    is a symplectic transformation, and it holds that 
    \[
    C_{\lambda}(K) = K_1',
    \]
    where $K_1' + K_2' = \A_{d',h'}[3]$ is the canonical decomposition.
\end{lemma}

\begin{proof}
    First note that $C_\lambda = C_{a_0}^{-a} \circ C_{a_1}^{-b} \circ C_{a_2}^{-c}$ with $a_0, a_1,a_2$ as in Proposition \ref{prop:symplectic-transformation-1}, hence $C_\lambda$ defines a symplectic transformation. Furthermore, it follows from the same proposition that the transformation $C_\lambda$ is represented by the matrix
    \[
    \begin{pmatrix}
        1 & 0 & -a & -b\\
        0 & 1 & -b & -c\\
        0 & 0 & 1 & 0\\
        0 & 0 & 0 & 1
    \end{pmatrix}
    \]
    with respect to the canonical bases of $\A_{d,h}$ and $\A_{d',h'}$, respectively.
\end{proof}

Note that the theorem requires as input a kernel of the form $K_2$, whereas the transformation in Lemma \ref{lem:trafo-kernel} constructs a kernel of the form $K_1$. This can be achieved by simply applying a discrete Fourier transform and multiplying by $[-1]$ (cf. Lemma \ref{lem:DFT-3torsion}).  The overall procedure is summarized in the next corollary.

\begin{corollary} \label{cor:isogeny-arbitrary}
    Let $\A_{d,h}$ be a p.p. abelian surface in Hesse form, and $(P_1,P_2,Q_1,Q_2)$ the canonical basis of $\A_{d,h}[3]$. Given a maximal isotropic subgroup $K = \langle P_1 + aQ_1 + bQ_2, P_2 + bQ_1 + cQ_2\rangle \subset \A_{d,h}[3]$, and $9$-torsion points $R_1,R_2$ satisfying $3 \cdot R_1 = P_1 + aQ_1 + bQ_2$ and $3\cdot R_2 = P_2 + bQ_1 + cQ_2$. Further let $C_\lambda$ be as in Lemma  \ref{lem:trafo-kernel}. Then the $(3,3)$-isogeny with kernel $K$ can be computed using the formulas from Theorem  \ref{thm:3-isogeny-dim2} with input $S_K(\A_{d,h})$ and $S_K(R_1), S_K(R_2)$, where $S_K = [-1] \circ M_{3,3} \circ C_\lambda$.
\end{corollary}

\subsection{Non-generic isogenies}\label{subsec:nongeneric}

In Theorem \ref{thm:3-isogeny-dim2}, it is assumed that the domain of the isogeny $\A_{d,h}$ is not exceptional (cf. Definition \ref{def:exceptional}). This property was essential to derive  the concrete scalars for the map $C_\lambda$. 

Here, we provide an alternative method to derive the scalars for the map $C_{\lambda}$ which also works in the exceptional case unless the codomain admits extra  automorphisms of order $3$. Apart from the $9$-torsion points $R_1$ and $R_2$, these formulas further require the coordinates of $R_1 + R_2$ and $R_1 - R_2$ as input.

\begin{lemma} \label{lem:exceptional-points}
    Let $\A_{d,h}$ be an exceptional p.p. abelian surface. Let $\tilde{t}=(\tilde t_0: \dots: \tilde t_4) $ as in  Definition \ref{def:exceptional}. Then there exist at least two indices $i \in \{0, \dots, 4\}$  with $\tilde t_i \neq 0$.
\end{lemma}

\begin{proof}
	Let $\A_{d,h} \to \A_{d',h'}$ be an isogeny as in Lemma \ref{lem:3-isogeny-form}, and denote by $t' = (t_0 :\dots :t_4')$, the coordinates defining the neutral element of the codomain. Then $t_i' = \lambda_i \cdot \tilde{t}_i$ for each $i \in \{0, \dots, 4\}$, where $\lambda = (\lambda_0: \dots :\lambda_4)$ is the scaling vector defining $C_\lambda$. The statement now follows from Corollary \ref{cor:non-zero-coordinates} which states that there are at least two indices $i \neq j \in \{0,\dots, 4\}$ so that $t_i', t_j'$ are nonzero.
\end{proof}

\begin{proposition} \label{prop:isogeny-formula-exceptional}
    Let $\A_{d,h}$ be a p.p. abelian surface in Hesse form. Denote $(P_1,P_2,Q_1,Q_2)$ for the canonical $3$-torsion basis, and let 
    \(
    \langle R_1,R_2\rangle \subset \A_{d,h}[9]
    \) be a maximal isotropic subgroup with $3 \cdot R_1  =  Q_1$, $3 \cdot R_2  =  Q_2$  and $R_3 = R_1 + R_2$, $R_4 = R_1 - R_2$. %We denote
    %\[
    %(x_{i,0}:\dots: x_{i,8}) = M \circ \cube(R_i) \quad \text{for } i =1,2,3,4.
    %\]
    
    Further, let $\tilde{t} = (\tilde{t}_0: \dots: \tilde{t}_4)$ as in Definition \ref{def:exceptional}. If there are two indices $i \neq j >0$ with $\tilde{t}_i, \tilde{t}_j \neq 0$, then there exists a unique $\lambda \in \PP^4$ so that 
    \[
    \phi: \A_{d,h} \to \A_{d',h'}, \quad P \mapsto C_\lambda \circ M \circ \cube(P)
    \]
    defines a $(3,3)$-isogeny with kernel $K_2 = \langle Q_1, Q_2\rangle$, and the canonical symplectic basis for $\A_{d',h'}[3]$ is given by  $(P_1',P_2',Q_1',Q_2')$, where $P_1' = \phi(R_1), \, P_2' = \phi(R_2)$ and $Q_1' = [-1]\phi(P_1), \, Q_2' = [-1]\phi(P_2)$. 
    
    Moreover, the coordinates of $\lambda$ can be read off from the $9$-torsion points $R_1,R_2,R_3,R_4$.
\end{proposition}

\begin{proof}
	We use the same notation as in the proof of Theorem \ref{thm:3-isogeny-dim2}. That is we decompose the isogeny $A \xrightarrow{\phi_2} A_2 \xrightarrow{C_\lambda} A_{d',h'}$ for some $\lambda \in \PP^4$. Let $(P_1',P_2',Q_1',Q_2')$ be the canonical symplectic basis of $\A_{d',h'}$. With the same arguments as in the proof of the theorem, we may assume that $\phi(R_1) = P_1'$ and $\phi(R_2)=P_2'$, hence $\phi(R_3) = P_1'+P_2' =: P_3'$ and $\phi(R_4) = P_1'-P_2' =:P_4'$.

	We denote
	\[
	\phi_2(R_i) = (x_{i,0}:\dots: x_{i,8}) \quad \text{for } i =1,2,3,4.
	\]
	The condition $C_\lambda(\phi_2(R_i)) = P_i'$ for $i = 1,2,3,4$ provides us with four equalities. 
	
	\begin{align*}
		\left(\lambda_0 x_{1,0}: \lambda_1 x_{1,1} : \dots: \lambda_3 x_{1,8}\right) = &~
		(t_1' : t_1' : t_0':  t_3': t_4' : t_2' :  t_4' : t_3' : t_2'),\\
		\left(\lambda_0 x_{2,0}: \lambda_1 x_{2,1} : \dots: \lambda_3 x_{2,8}\right)  = &~
		(t_2': t_3' : t_4' : t_2' : t_4' : t_3 ': t_0' : t_1' : t_1' ),\\
		\left(\lambda_0 x_{3,0}: \lambda_1 x_{3,1} : \dots: \lambda_3 x_{3,8}\right)  = &~
		(t_3': t_4' : t_2' : t_4' : t_3' : t_2 ': t_1' : t_1' : t_0' ),\\
		\left(\lambda_0 x_{4,0}: \lambda_1 x_{4,1} : \dots: \lambda_3 x_{4,8}\right)  = &~
		(t_4': t_3' : t_2' : t_1' : t_1' : t_0 ': t_3' : t_4' : t_2' ).
	\end{align*}
	
	For each $i$, the condition $\tilde{t_i} \neq0$ is equivalent to $t_i' \neq 0$ (cf. the proof of Lemma \ref{lem:exceptional-points}). Now if $t_i'\neq 0$ for some $i>0$, then we can deduce one relation on the coordinates of $\lambda$ from each of the above equalities.\\
	Explicitly, if $t_1' \neq 0$, then
	\[
	(\lambda_0:\lambda_1) = (x_{1,1}:x_{1,0}), ~~
	(\lambda_4:\lambda_3) = (x_{2,8}:x_{2,7}), ~~
	(\lambda_2:\lambda_4) = (x_{3,7}:x_{3,6}), ~~
	(\lambda_2:\lambda_3) = (x_{4,4} : x_{4,3}).
	\]
	If $t_2' \neq 0$, then
	\[
	(\lambda_4:\lambda_3) = (x_{1,8}:x_{1,5}), ~~
	(\lambda_0:\lambda_2) = (x_{2,3}:x_{2,0}), ~~
	(\lambda_1:\lambda_4) = (x_{3,5}:x_{3,2}), ~~
	(\lambda_1:\lambda_3) = (x_{4,8} : x_{4,2}).
	\]
	If $t_3' \neq 0$, then
	\[
	(\lambda_2:\lambda_4) = (x_{1,7}:x_{1,3}), ~~
	(\lambda_1:\lambda_4) = (x_{2,5}:x_{2,1}), ~~
	(\lambda_0:\lambda_3) = (x_{3,4}:x_{3,0}), ~~
	(\lambda_1:\lambda_2) = (x_{4,6} : x_{4,1}).
	\]
	If $t_4' \neq 0$, then
	\[
	(\lambda_3:\lambda_2) = (x_{1,6}:x_{1,4}), ~~
	(\lambda_1:\lambda_3) = (x_{2,4}:x_{2,2}), ~~
	(\lambda_1:\lambda_2) = (x_{3,3}:x_{3,1}), ~~
	(\lambda_0:\lambda_4) = (x_{4,7} : x_{4,0}).
	\]
	
	Note that in all of the above cases, there is redundancy in the derived relations. In particular, using only one set of relations obtained for some $t_i' \neq 0$ is not enough to recover the coordinate vector $\lambda \in \PP^4$. However, two sets of relations combined for some $t_i',t_j' \neq 0$ with $i\neq j > 0$ is enough to recover $\lambda$. 
\end{proof}

\begin{remark}
	The only case that is missing in  Proposition \ref{prop:isogeny-formula-exceptional} is the case when $\tilde{t}$ has exactly two non-zero indices, one of which is $i=0$. Indeed, in this case the statement of the proposition does not hold. The codomain admits a special automorphism of order $3$, and as a result the scaling $C_\lambda$ is not uniquely defined by the images of the $9$-torsion points.
	
	To see this, let $0_{\A_{d',h'}} = (t_0': \dots : t_3') $ be the neutral element of the codomain. For simplicity assume $t_0',t_1' \neq0$ and $t_2' = t_3' =t_4'=0$. The extra automorphism is similar for the remaining cases.
	
	We note that  $t_0'^3 \neq -2t_1'^3, t_1'^3$, otherwise one can check that $\A_{d',h'}$ has a singularity at $0_{\A_{d',h'}}$. The conditions on $t_i'$ then imply
	\[
	d' = (t_0'^3+2t_1'^3 : 3t_0't_1'^2:0:0:0), \quad h' = (0:0:h_2':h_3':h_4')
	\] 
	for some $h_2',h_3',h_4'$. Now consider the map $C_\lambda$ with $\lambda = (1:1:\omega:\omega:\omega)$ (cf. Proposition \ref{prop:symplectic-transformation-1}).  This defines an automorphism of order $3$ of $\A_{d,h}$.
\end{remark}

    \section{Example} \label{sec:example}

To illustrate the main theorem and its applications, we provide a detailed example. In this example, we compute a $(9,9)$-isogeny between a product of elliptic curves. The isogeny computation is decomposed into a $(3,3)$-gluing and a $(3,3)$-splitting. The example illustrates the most important subtleties that can occur in the computation of $(3,3)$-isogenies. The example does not contain a generic $(3,3)$-isogeny computation, but we assure the reader that this is the most straightforward case. More details, and further examples can be found in our GitHub repository \cite{source_code}.

\subsection{Setting} We work over the finite field $\FF_{p^2}$  with $p = 269$. This field contains a primitive third root of unity $\omega \in \FF_{p^2}$. Let $E_1: y^2 = x^3 + 1$ and $E_2: y^2= x^3 + 66x + 134$ over $\FF_{p^2} \setminus \FF_p$. The two curves are connected by a $5$-isogeny $\phi: E_1 \to E_2$ with kernel $\ker(\phi) = \langle (8,128)\rangle$. And in particular, they fit into the following commutative diagram:
\[\begin{tikzcd}
    E_1 \arrow[r,"\phi"] \arrow[d, "{[2]}"] & E_2 \arrow[d, "{[2]}"]\\
    E_1 \arrow[r,"\phi"] & E_2
\end{tikzcd}\]
By Kani's lemma \cite{kani1997number}, this diagram induces a $(9,9)$-isogeny $\Phi: E_1 \times E_2 \to E_1 \times E_2$ with kernel
\[
\ker(\Phi) = \{ ([2]P, \phi(P)) \mid P \in E_1[9] \}.
\] 

This is the isogeny we wish to compute. Recall that our algorithm for the computation of $(3,3)$-isogenies requires certain $9$-torsion points as auxiliary input. Hence, we first fix a basis 
\[
E_1[27] = \left\langle P,Q \right\rangle \quad \text{with } P = (252,71),\, Q = (7\omega,13\omega+141),
\]
and fix the following notation
\[
P_1 = [2] \cdot P, \; Q_1 = [2] \cdot Q \in  E_1[27], \quad P_2 = \phi(Q), \; Q_2 = \phi(P) \in E_2[27].
\]
In particular, in this notation, $\ker(\Phi) = [3] \cdot \langle (P_1,Q_2), (Q_1,P_2) \rangle$.

\subsection{Transformation to Hesse form}
To apply our algorithm, it is necessary to translate the input to elliptic curves in Hesse form. This transformation is linear, and it is determined by specifying the image of the $3$-torsion elements. Concretely, %one can check that $E_1$ and $E_2$ are respectively isomorphic to 
%\[
%H_{(-2:1)}: X^3 + Y^3 + Z^3 = -6 XYZ, \quad H_{(18\omega^2:1)}:X^3+Y^3+Z^3 = 18\omega^2 XYZ,
%\]
%and
the transformations are given by
\[
f_1: E_1 \xrightarrow{\begin{pmatrix} 1 & 0 & 231 \omega \\
1 & 19 \omega & 19 \omega \\
1 & 250 \omega & 19 \omega
 \end{pmatrix}} \H_{(-2:1)},
 \quad 
f_2:  E_2 \xrightarrow{\begin{pmatrix} 1 & 0 & 179 \\
-9\omega^2 & 9-9\omega & 49 \omega^2 \\
-9\omega^2 & -9-9\omega& 49 \omega^2
 \end{pmatrix}} \H_{(-18\omega^2:1)}.
\]
These transformations are chosen so that
\begin{align*}
&f_1([9]P_1) = (-1:1:0), \; f_1([9]Q_1) = (0:-\omega^2:1),\\
&f_2([9]P_2) = (-1:1:0), \; f_2([9]Q_2) = (0:-\omega^2:1).
\end{align*}
In our implementation in SageMath, an elliptic curve in Hesse form is constructed by calling
\texttt{H1 = EllipticCurveHessianForm(E1, basis=[9*P1,9*Q1])}, 
where \texttt{E1} is the elliptic curve in Weierstra\ss\ form, and \texttt{P1}, \texttt{Q1} are the points $P_1,Q_1$ defined above. The matrix describing the corresponding transformation can be accessed as \texttt{H1.\_trafo}. Similarly, the second elliptic curve in Hesse form is constructed by \texttt{H2 = EllipticCurveHessianForm(E2, basis=[9*P2,9*Q2])}. 

As described in Subsection \ref{subsec:product}, the image of the Segre embedding of two elliptic curves in Hesse form provides us with an abelian surface in Hesse form. Applying Lemma \ref{lem:product-equation} we find
\[
\iota_S\left(\H_{(18\omega^2:1)} \times \H_{(-2:1)} \right) = \A_{d,h},\quad \text{with } d = (-36\omega^2: 18\omega^2: -2: 1: 1), \; h =  (0: 0: 0: -1: 1).
\]
In the implementation, the surface $\A_{d,h}$ is constructed as \texttt{A0 = AbelianSurfaceHessianForm([H2,H1])}.

\subsection{General strategy}
With the above information as input, the isogeny $\Phi$ can now be computed using the function \texttt{compute\_isogeny\_chain}. The intermediate computations performed by this function can be summarized as follows:
\[
\begin{tikzcd}
    \H_{(18\omega^2:1)} \times \H_{(-2:1)} \arrow[r,"\iota_S"] &\A_{d,h} \arrow[d, swap, "S_K"] & \A_{d',h'} \arrow[d,swap, "S_K{'}"]& \A_{d'',h''} \arrow[d, swap, "S_{K''}"]&\\
    & \A_{\tilde d, \tilde h} \arrow[ur, "\Phi_1"] &  \A_{\tilde d', \tilde h'} \arrow[ur, "\Phi_2"] &  \A_{\tilde d'', \tilde h''} \arrow[r, "\iota_S^{-1}"]&  \H_{(-2:1)} \times \H_{(18\omega^2:1)}.
\end{tikzcd}
\]
The maps $S_K,S_K',S_K''$ are symplectic transformations and $\Phi_1, \Phi_2$ are $(3,3)$-isogenies computed using the formulas from Theorem \ref{thm:3-isogeny-dim2}. Below, we provide more details on the computations. We note that for a general isogeny chain computation in the framework, all intermediate symplectic transformations are of the form $[-1] \circ M_{3,3}$ (this is the case for $S_{K'}$ here), whereas in the first and last step, the symplectic transformation depends on the setup (this is the case for $S_K$ and $S_{K''}$ here).

\subsection{Gluing isogeny}
As per Corollary \ref{cor:product-basis}, the canonical basis of $\A_{d,h}[3]$ is given by 
\[
\left(P_1^0,P_2^0, Q_1^0,Q_2^0\right) = \left([9] \iota_S(0, P_1),\; [9] \iota_S(P_2,0), \; [9] \iota_S(0,Q_1), \; [9] \iota_S(Q_2,0) \right),
\]
and $\ker(\Phi)[3] = \langle Q_1^0 + P_2^0, Q_2^0+P_1^0\rangle$. In order to apply Theorem \ref{thm:3-isogeny-dim2}, we need to apply a symplectic transformation that transforms $K:=\ker(\Phi)[3]$ to a group of the form $K_2$. As described in Corollary \ref{cor:isogeny-arbitrary}, this is achieved by the transformation \[
S_K := [-1] \circ M_{3,3} \circ \C_\lambda,\quad \text{with } \lambda = (1:1:1:\omega:\omega^2).
\] Note that the value for $\lambda$ is deduced from Lemma \ref{lem:trafo-kernel}. Explicitly, the coefficients defining the codomain surface of $S_K: \A_{d,h} \to \A_{\tilde{d},\tilde{h}}$ are given by
\[
 \tilde d = (263: 54\omega + 51 : 3 : 57\omega + 57 : 51\omega + 54),\; \tilde h =  (-\omega: \omega: \omega: \omega^2: 1).
\]

With the notation of Eq. \ref{eq:neutral-element}, the neutral element is defined by the vector $t= (-2\omega:\omega:\omega:\omega^2:1)$, and one can check that $\A_{\tilde d, \tilde h}$ is not exceptional. 
Now we can apply the main theorem to compute the $(3,3)$-isogeny
\[
\Phi_1: \A_{\tilde d, \tilde h} \to \A_{d',h'}, \quad P \mapsto C_{\lambda'} \circ M \circ \cube(P)
\]
with
\[
\lambda' = (88: 170: 181\omega + 5: 88\omega + 93: 1),
\]
and 
\[
d' = (107: 144 : 169\omega + 200: 100\omega + 31:1),\quad h' = (104: 76: 203\omega + 167: 66\omega + 233: 1).
\]
One can check that $h'$ is a {\em nonsingular} point on the Burkhardt quartic. This means that $\A_{d',h'}$ is irreducible, and the $(3,3)$-isogeny $\Phi_1$ is indeed a gluing. 

\subsection{Splitting isogeny}
By Theorem \ref{thm:3-isogeny-dim2}, we have that $K':=\Phi_1(S_K(\ker(\Phi))[3] = K_1'$, where $\A_{d',h'} = K_1' + K_2'$ is the symplectic decomposition of the $3$-torsion. This is transformed into a kernel of the form $K_2$ by applying a discrete Fourier transform (Lemma \ref{lem:DFT-3torsion}). Moreover, a multiplication by $[-1]$ is necessary to transform the generators into the correct form. Denote $S_{K'} = [-1] \circ M_{3,3}$ and $\A_{\tilde d',\tilde h'}$ for the codomain of $S_{K'}$. Then, one may again apply Theorem \ref{thm:3-isogeny-dim2} to compute the splitting isogeny as
\[
\Phi_2: \A_{\tilde d',\tilde h'} \to \A_{d'',h''}, \quad P \mapsto  C_{\lambda'} \circ M \circ \cube(P),
\]
where 
\[
\lambda'' = (144\omega:  250\omega: 126\omega + 260: 143\omega + 134: 1)
\]
and 
\[d'' = (99\omega + 99: 85\omega + 85: 100\omega + 184: 169\omega + 84: 1)
 ,\quad h'' = (-\omega: \omega: \omega:\omega^2: 1)
.\]

As expected, $h''$ is a singular point of the Burkhardt quartic, hence $\Phi_2$ is indeed a $(3,3)$-splitting. In order to recover the two elliptic curves, one needs to apply another symplectic transformation to obtain a model equipped with the product polarization (Definition \ref{def:product-structure}). In our implementation, this transformation is computed using \texttt{product\_structure\_transformation}. This yields
\[
S_{K''}: \A_{d'',h''} \to \A_{\tilde d'',\tilde h''}, \quad P \mapsto C_{(1:1:1:\omega:\omega^2)} \circ M_{3,3} \circ C_{(1:1:1:\omega:\omega^2)} (P)
\]
with 
\[
\tilde d'' = (-36\omega^2: -2: 18 \omega^2: 1: 1)
,\quad \tilde h'' = (0:0:0:-1:1).
\]
And in particular,
\[
\A_{\tilde d'',\tilde h''} =  \iota_S\left(\H_{(-2:1)} \times \H_{(18\omega^2:1)} \right).
\]
Note that the codomain of the isogeny $\Phi$ was already clear from the construction in this example.

    \section{Generalization of the approach} \label{sec:generalizations}

The main ingredient for our new isogeny formula is Lemma \ref{lem:3-isogeny-form} which states that for some appropriate scaling vector $\lambda$, the map
\[
\phi: \A_{d,h} \to \A_{d',h'}, \quad P \mapsto C_\lambda \circ M \circ \cube(P)
\] 
defines a $(3,3)$-isogeny with kernel $K_2$. It is only natural to ask, whether this approach can be generalized to higher dimensional abelian varieties defined by a symmetric theta structure of level $3$, and whether it further generalizes to different levels.

\subsection{Going up in dimension}
To answer the first question, let $A$ be a p.p. abelian variety of dimension $g$. Consider a line bundle $L$ which determines the principal polarization of $A$. Then $L^{\otimes 3}$ is very ample by Lefschetz's Theorem; it defines an embedding $A \hookrightarrow \PP^{3^g-1}$ and one may choose an embedding so that the action of the three-torsion group is normalized. In analogy with our definition in dimension $2$, we say that $A$ is in {\em Hesse form}. Such an embedding is induced by a symmetric theta structure and the coordinate functions correspond to level-$3$ theta functions. In the following, we explain how the steps in the proof of Lemma \ref{lem:3-isogeny-form} generalize to arbitrary dimension. 

Similar as in Definition \ref{def:symplectic-decomposition}, one can define a canonical symplectic basis
\[
(P_1, \dots, P_g, Q_1,\dots, Q_g) 
\]
for $A[3]$ so that the action of the points $P_i$ is given by permutations of order $3$, and the action of $Q_i$ is given by scalings with a third root of unity (cf. \cite{mumford1966equations}). It follows that also in this general setting, the map defined by coordinate-wise cubing defines an isogeny $\cube:A \to A_1$ with kernel $K_2 = \langle Q_1, \dots, Q_g\rangle$. 

The discrete Fourier transform with respect to a third root of unity $\omega$ may be defined inductively:
\[
M_{3^g} =  \begin{pmatrix}
		M_{3^{g-1}} & M_{3^{g-1}} & M_{3^{g-1}}\\
		M_{3^{g-1}} & \omega M_{{3^{g-1}}} & \omega^2 M_{3^{g-1}}\\
		M_{3^{g-1}} & \omega^2M_{3^{g-1}} & \omega M_{3^{g-1}}
	\end{pmatrix} \quad \text{for } g>1.
\]

Up to reordering the generators of the basis $(P_1, \dots, P_g, Q_1,\dots, Q_g)$, one can then check that $M_{3^g} \cdot P_i = -Q_i$ and $M_{3^g}\cdot Q_i = P_i$ (as in Lemma \ref{lem:DFT-3torsion}). In particular, we see that on the codomain $A_2$ of 
\[
\phi_2 = (M_{3^g} \circ \cube): A \to A_2,
\]
the action of the points $\phi_2(P_i)$ is in canonical form. In the language of Section \ref{sec:twisted-surface}, one could say that $A_2$ is in twisted Hesse form. In particular, there exists a scaling $C_\lambda$ with $\lambda \in \PP^{(3^g-1)/2}$ to transform $A_2$ into  Hesse form. 

We expect that apart from certain edge cases, it will also be possible to determine the scaling vector $\lambda$ from a given maximal 9-isotropic subgroup $\langle R_1, \dots, R_g\rangle$ above the kernel of the isogeny. However, we did not work out the general formulas in detail. 

\begin{remark}
	In the preparation of this manuscript, we further made this idea explicit in the case of elliptic curves. Isogeny formulas for elliptic curves in Hesse form are well known, and we verified that our new approach yields the same result. Interestingly, this interpretation of $3$-isogenies led to the discovery of some new tripling formulas on the Kummer line of elliptic curves in {\em twisted} Hesse form. We presented these results at a cryptography conference, and the article is currently in submission to the proceedings of this conference \cite{hessiantripling}. 
\end{remark}

\subsection{Different levels $\ell$}

We point out that our ideas could also lead to new isogeny formulas for different levels $\ell$. Indeed, as explained in the introduction, the inspiration for our work comes from the level-$2$ theta setting, where particularly efficient $2$-isogeny formulas are known. In this case, the isogeny formulas follow from the classical duplication formulas of theta functions, and hold in arbitrary dimension $g$. 

\begin{remark} \label{rem:generalization-level}
	 To formulate a generalization of Lemma \ref{lem:3-isogeny-form} to arbitrary level $\ell>2$, we introduce the following notation:
	\begin{itemize}
		\item $\odot^\ell$:  coordinate-wise exponentiation with exponent $\ell$, i.e. $\odot^\ell(x) = (x_0^\ell : \dots : x_{3^g-1}^\ell)$.
		\item $M_{\ell^g}$: The discrete Fourier transform with respect to a primitive $\ell$-th root of unity and dimension $g$.
		\item $C_{\lambda}$ with $\lambda \in \PP^{{(\ell^{g}-1)}/2}$: coordinate-wise scaling, symmetric with respect to multiplication by $[-1]$. 
	\end{itemize}
	
	Now let $A$ be a p.p. abelian variety, and $\ell>2$ a prime. Further, let 
	\(
	\Theta_\ell: A \to \PP^{\ell^{g}-1}
	\)
	be the embedding induced by a symmetric theta structure of level $\ell$. 
	We claim that there exists a scaling vector $\lambda \in \PP^{{(\ell^{g}-1)}/2}$ so that the map
	\[
	\phi = (C_\lambda \circ M_{\ell^g} \circ \odot^\ell):  \Theta_\ell(A) \to {A}'
	\]
	defines an isogeny with maximal $\ell$-isotropic kernel; and the codomain ${A}'$ is induced by a symmetric theta structure $\Theta_\ell'$ of level $\ell$ as well.	We describe the idea briefly.
    
    First, one fixes a canonical symplectic basis  $(P_1, \dots, P_g, Q_1, \dots, Q_g)$ for $A[\ell]$. Similar as in the setting of abelian varieties in Hesse form, addition by a point $P_i$ acts by certain permutations of order $\ell$; and addition by a point $Q_i$ acts by scaling with $\ell$-th roots of unity. In particular, the map $\odot^\ell$ defines an isogeny with kernel $K_2 = \langle Q_1, \dots, Q_g\rangle$. And as in the proof of Lemma \ref{lem:3-isogeny-form}, there must exist a linear transformation so that the codomain is in the correct form. To describe this transformation explicitly, the key is to observe that the images of the points $P_1, \dots, P_g $ under $\phi_2 := M_{\ell^g}\circ \odot^{\ell}$ act by scalings with $\ell$-th roots of unity on the points of the codomain of $\phi_2$.
\end{remark}

\begin{remark}
    Note that the number of coordinates to represent a point on an abelian variety grows with the level $\ell$. It would be interesting to see if the above method could still lead to more efficient isogeny formulas. Even in the case of elliptic curves, this could lead to new insights on isogenies. The next easiest case to consider is an elliptic curve defined as the intersection of two quadrics in $\PP^3$. With a suitable choice of coordinates, this corresponds to a level-$4$ theta structure (cf. \cite[\S 5, Case (c)]{mumford1966equations}). This model should therefore be suitable for the computation of $4$-isogenies.
\end{remark}
	
	%\bibliographystyle{plain}	
	%\bibliography{ref.bib}

\end{document}